\documentclass{amsart}
\usepackage{graphicx}
\usepackage{amssymb}
\usepackage[hyphens]{url} 
\usepackage{amsmath}
\usepackage{cite}
\usepackage{mathrsfs}
\usepackage{hyperref}
\usepackage{amsthm}
\usepackage{tikz-cd}
\usepackage{caption}
\usepackage{float}
\usepackage[utf8]{inputenc}
\usepackage{graphicx}
\usepackage{stmaryrd}
\usepackage{tikz}
\usepackage{tikz-cd} 
\usepackage[all]{xy}
\usepackage{comment}
\usepackage{bm}
\usepackage{comment}
\usepackage{amsmath,amsfonts,amssymb}
\usepackage{tcolorbox}

\newcommand{\Z}{\mathbb{Z}}
\newcommand{\vs}{\vspace{10pt}}
\DeclareMathOperator{\Aut}{Aut}
\DeclareMathOperator{\Ad}{Ad}

\usepackage[a4paper]{geometry}
\usepackage{enumitem}
\usepackage{xcolor}
\usepackage{nicefrac}
\hypersetup{
	colorlinks=true,
	linkcolor={blue},
	citecolor={blue},
	urlcolor={blue}}

\usepackage{cleveref}

\theoremstyle{plain}

\newtheorem{thm}{Theorem}[section]
\newtheorem{prop}[thm]{Proposition}

\newtheorem{claim}[thm]{Claim}
\newtheorem{qu}[thm]{Question}
\newtheorem{lemma}[thm]{Lemma}

\theoremstyle{definition}

\newtheorem{defn}[thm]{Definition}

\newtheorem{ex}[thm]{Example}

\theoremstyle{remark}
\newtheorem{rmk}[thm]{Remark}
\newlength{\plarg}
\usetikzlibrary{cd}

\title[Conjugacy Problem for Dehn Twists of Free Products of Free Abelian Groups]{Conjugacy Problem for Dehn Twists of Free Products of Finitely Generated Free Abelian Groups}

\author{Amir Y. Weiss Behar \and Chris Karpinski \and Bratati Som}

\begin{document}
\maketitle
\begin{abstract}
     We show solubility of the conjugacy problem for Dehn twist automorphisms of finitely generated free products of free abelian groups.
\end{abstract}

\section{Introduction}

Two automorphisms $\phi_1, \phi_2$ of a group $G$ are said to be \emph{conjugate} if there exists an automorphism $\psi$ of $G$ such that $\phi_2 = \psi \phi_1 \psi^{-1}$. The \emph{conjugacy problem} for automorphisms of a group $G$ is the problem of \emph{deciding}, given a presentation for $G$, whether or not given automorphisms $\phi_1, \phi_2$ of $G$ are conjugate. By deciding, we mean finding an algorithm that takes as input a given presentation of $G$ and given automorphisms $\phi_1, \phi_2$ of $G$ and outputs ``yes'' or ``no'' depending on whether or not $\phi_1, \phi_2$ are conjugate. 

In this paper, we study the conjugacy problem for \emph{Dehn twist} automorphisms of free products of free abelian groups. Dehn twist automorphisms of a group $G$ are algebraic generalizations of Dehn twist homeomorphisms of a surface, and are defined in terms of splittings of $G$ either as an amalgamated free product (these are the \emph{separating} Dehn twists) or splittings of $G$ as an HNN extension (these are the \emph{non-separating} Dehn twists). See Section \ref{section: Dehn twists} for the definition of Dehn twists. The conjugacy problem for Dehn twists of finitely generated free groups was solved in \cite{Cohen_Lustig} (where a more general definition of Dehn twists was considered). We generalize the results of \cite{Cohen_Lustig} from free groups to the broader class of finitely generated free products of free abelian groups. 

\begin{thm}
    \label{thm: main theorem}
    Let $G$ be a free product of finitely many finitely generated free abelian groups. There is an algorithm that, given a finite presentation of $G$ and two Dehn twist automorphisms $\delta_1, \delta_2$ of $G$, decides whether or not $\delta_1$ and $\delta_2$ are conjugate in $\Aut(G)$. 
\end{thm}

The main technique employed in our paper is the study of the \emph{mapping torus} of an automorphism $\phi$, which is the semi-direct product $\Gamma_{\phi} := G \rtimes_{\phi} \Z = \langle G, t \vert tgt^{-1} = \phi(g) \text{ for all } g \in G \rangle$. Given a Dehn twist automorphism $\phi$, we construct a \emph{canonical} tree on which $\Gamma_{\phi}$ acts on, using which we deduce easy to check, necessary and sufficient algebraic conditions for two Dehn twist automorphisms to be conjugate (see Lemma \ref{lem:upgrade_to_fiber_preserving} and Lemma \ref{lem: fibre and orientation preserving iso}). Using these algebraic conditions, we develop our algorithm to decide if two given Dehn twist automorphisms are conjugate. 

The class of free products of free abelian groups sits inside the larger classes of \emph{CSA (conjugacy separated abelian)} groups (which are groups where every maximal abelian subgroup is malnormal) as well as right-angled Artin groups. It remains open whether the solubility of the conjugacy problem for Dehn twist automorphisms extends to the classes of finitely generated CSA groups and right-angled Artin groups. 

\begin{qu}
    Is the conjugacy problem solvable for Dehn twist automorphisms of finitely generated CSA groups? 
\end{qu}
\begin{qu}
    Is the conjugacy problem solvable for Dehn twist automorphisms of finitely generated right-angled Artin groups? 
\end{qu}

Dehn twists can also be defined for any splitting of a group $G$ with abelian edge groups (see \cite[Definition 5.1]{Cohen_Lustig}). Our work also does not address conjugacy of such more general Dehn twists. 
\begin{qu}
    Is the conjugacy problem solvable for general Dehn twist automorphisms (as defined in \cite[Definition 5.1]{Cohen_Lustig}) of finitely generated free products of free abelian groups? 
\end{qu}

\textbf{Acknowledgments.}
We thank François Dahmani and Nicholas Touikan for suggesting the main question addressed in this paper. We are especially grateful to Nicholas Touikan for his guidance and for many insightful discussions. We also thank Antoine Goldsborough and Kuwari Mahanta for helpful discussions. This research was supported in part by the International Centre for Theoretical Sciences (ICTS) for participating in the program: Geometry in Groups (code: ICTS/GIG2024/07). We thank the organizers for facilitating this program.

\section{Preliminaries}

\subsection{Groups acting on trees and Bass--Serre theory}

\begin{defn}
    Given a group $\Gamma$, a \textbf{$\Gamma$-tree} is a tree $T$ with an action of $\Gamma$ on $T$ by isometries.

    A \textbf{morphism} of $\Gamma$-trees $T,S$ is a continuous $\Gamma$-equivariant map $f : T \to S$.

    An \textbf{isomorphism} of $\Gamma$-trees $T,S$ is a morphism $f : T \to S$ of $\Gamma$-trees whose inverse $f^{-1} : S \to T$ is also a morphism of $\Gamma$-trees. 
\end{defn}

\begin{defn}
    A $\Gamma$-tree $T$ is \textbf{non-trivial} if it is not a single vertex and is \textbf{minimal} if the action $\Gamma \curvearrowright T$ is minimal, i.e.\ there is no proper $\Gamma$-invariant subtree. 
\end{defn}

\begin{defn}
    An isometry $\gamma$ of a tree $T$ is \textbf{elliptic} if $\gamma$ fixes a vertex of $T$. Otherwise, $\gamma$ is called \textbf{hyperbolic}. 
    An action of a group $\Gamma$ by isometries on a tree $T$ is \textbf{elliptic} if $\Gamma$ fixes a vertex of $T$.
\end{defn}

By \cite[Theorem 2.1(2)]{Introduction_to_group_theory}, an isometry $\gamma$ of a tree $T$ is hyperbolic if and only if it translates along a bi-infinite geodesic line $l \subset T$, called the \textbf{axis} of $\gamma$. Furthermore, by \cite[Proposition 2.5]{Introduction_to_group_theory}, if a finitely generated group $\Gamma$ acts by isometries on a tree $T$ such that every element is elliptic, then $\Gamma$ acts elliptically on $T$. It follows that if a finitely generated group $\Gamma$ of isometries of $T$ does not act elliptically, then it must have an element which acts as a hyperbolic isometry. 

\subsection{Graphs of groups and Bass--Serre theory}

We assume the reader is familiar with Bass--Serre theory. Here, we simply set notation and definitions that we will use in this paper. For background on Bass--Serre theory, see, for instance \cite{Trees}. 

A \textbf{graph} $X = (V, E, \iota, \tau, \overline{})$ consists of sets $V$ of vertices and $E$ of (oriented) edges, maps $\iota : E \to V$ and $\tau : E \to V$ which we call the \textbf{origin} and \textbf{terminus} maps, respectively, together with a fixed point free involution $\bar{} : E \to E$ that assigns to every edge $e \in E$ the \emph{reverse} edge $\overline{e} \in E$, which satisfies $\iota(\overline{e}) = \tau(e)$ and $\tau(\overline{e}) = \iota(e)$. 

Note that we allow our graphs to have loops (i.e.\ edges $e$ such that $\iota(e) = \tau(e)$) and bigons (i.e.\ pairs of edges $e_1, e_2$ such that $\iota(e_1) = \iota(e_2)$ and $\tau(e_1) = \tau(e_2)$. 

A \textbf{graph of groups} $\mathcal{G} = (X, (G_v)_{v \in V}, (G_e)_{e \in E}, (\iota_e)_{e \in E}, (\tau_e)_{e \in E})$ consists of a graph $X$ (the \emph{underlying graph} of $\mathcal{G}$), groups $G_v$ and $G_e$ for each vertex $v \in V$ and each edge $e \in E$ and injective group homomorphisms $\iota_e : G_e \hookrightarrow G_{\iota(e)}$ and $\tau_e : G_e \hookrightarrow G_{\tau(e)}$ for every edge $e$, satisfying $G_e = G_{\overline{e}}$ and $\iota_e = \tau_{\overline{e}}$ for each edge $e$. We refer to the groups $G_v$ and $G_e$ as the \emph{vertex} and \emph{edge} groups of $\mathcal{G}$, respectively. 

Given a graph of groups $\mathcal{G}$ as above and a spanning (or maximal) tree $\mathcal{T} \subset X$, the \textbf{fundamental group} of $\mathcal{G}$ with respect to $\mathcal{T}$ is the group: 

$$\pi_1(\mathcal{G}, \mathcal{T}) = ((*_{v \in V} G_v) * F(E)) / \langle \langle e\overline{e}: e \in E, e : e \in E(\mathcal{T}), e \tau_e(g) e^{-1}\iota_e(g)^{-1} \text{ for all $g \in G_e$} \rangle \rangle $$

i.e.\ $\pi_1(\mathcal{G}, \mathcal{T})$ is the group generated by the vertex groups $G_v$ for each $v \in V$ and edges $e \in E$ subject to the relations: 

\begin{itemize}
    \item $e^{-1} = \overline{e}$ for each $e \in E$,
    \item $e = 1$ for each $e \in E(\mathcal{T})$,
    \item $e \tau_e(g) e^{-1} = \iota_e(g)$ for each $e \in E$ and each $g \in G_e$
\end{itemize}

By \cite[Proposition 20]{Trees}, $\pi_1(\mathcal{G}, \mathcal{T})$ is independent of the spanning tree $\mathcal{T}$ up to isomorphism, so we can simply write $\pi_1(\mathcal{G})$ for the fundamental group of $\mathcal{G}$. 

Given a graph of groups $\mathcal{G}$, there is a tree $T$, called the \textbf{Bass--Serre tree} of $\mathcal{G}$, with an action of $G = \pi_1 (\mathcal{G})$ such that the vertex and edge stabilizers of $T$ are conjugates of vertex and edge groups in $\mathcal{G}$. Conversely, given a group $G$ acting on a tree $T$ without inversions (i.e.\ if $g \in G$ fixes an edge $e$ of $T$, then $g$ fixes the endpoints of $e$), we have an induced \emph{quotient graph of groups} $\mathcal{G}$ (with underlying graph $ X = T/G$) whose Bass--Serre tree is $T$. The fundamental theorem of Bass--Serre theory gives a correspondence between graphs of groups $\mathcal{G}$ with a given fundamental group $G$ and $G$-trees $T$, where the corresponding tree of $\mathcal{G}$ is its Bass--Serre tree and the corresponding graph of groups of a $G$-tree $T$ is the quotient graph of groups $T/G$. 

In this paper, we will mostly focus on the case when $X$ is a single edge (in which case $\pi_1(\mathcal{G})$ is an \textit{amalgamated free product}) or when $X$ is a loop (in which case $\pi_1(\mathcal{G})$ is an \textit{HNN extension}). 

Recall that an \textbf{amalgamated free product} of two groups $A,B$ along a common subgroup $C$ (with inclusions $\iota_A : C \hookrightarrow A$ and $\iota_B : C \hookrightarrow B$ of $C$ into $A,B$) is the group

$$A *_C B = (A * B) / \langle \langle \iota_A(c)\iota_B(c)^{-1} : c \in C \rangle \rangle$$

i.e.\ $A *_C B$ is the group generated by $A$ and $B$ with the additional relations that every element $c$ in $C$ as an element of $A$ is identified with its copy in $B$. 

Recall that an \textbf{HNN extension} of a group $A$ along isomorphic subgroups $D, C$ inside $A$ (with $\sigma: D \to C$ an isomorphism) is the group 

$$A_{* D^e = C} = (A * \langle e \rangle)/ \langle \langle e d e^{-1} \sigma(d)^{-1} \text{ for all $d \in D$} \rangle \rangle$$

i.e.\ $A_{* D^e = C}$ is the group generated by $A$ and $e$ with the additional relations that every element $d$ in $D$ is conjugated by $e$ to its image $\sigma(d)$ in $C$. 

In the case of when $\mathcal{G}$ is a segment with vertex groups $A$ and $B$, edge group $C$ and $G=\pi_1(\mathcal{G}) = A *_C B$, the Bass--Serre tree of $\mathcal{G}$ is the tree with vertices left cosets $gA, gB$ of $A$ and $B$ in $G$, with vertices $gA, gB$ connected by an edge labeled $gC$. 

In the case of when $\mathcal{G}$ is a loop with vertex group $A$, edge groups $D \cong C$ in $A$ and $G=\pi_1(\mathcal{G}) = A *_{D^e = C}$, the Bass--Serre tree of $\mathcal{G}$ is the tree with vertices left cosets $gA$ and vertices $gA$ and $geA$ connected by an edge $gC$. 

\subsection{Acylindrical actions on trees}

\begin{defn}
    \label{defn: acylindrical action on a tree}
    An action of a torsion-free group $\Gamma$ on a tree $T$ (or the tree $T$ itself) is $k$-\textbf{acylindrical} for $k \geq 0$ if the pointwise stabilizer of every geodesic segment of length $k+1$ in $T$ is trivial. The action is \textbf{acylindrical} if it is $k$-acylindrical for some $k \geq 0$. 
\end{defn}

\begin{defn}
    \label{defn: non-elementary acylindrical}
    An acylindrical action of a group $\Gamma$ on a tree $T$ is \textbf{non-elementary} if the action is not elliptic and $\Gamma$ is not virtually cyclic. 
\end{defn}

Groups admitting non-elementary acylindrical actions on trees (or more generally, on Gromov hyperbolic metric spaces) have many nice properties. See, for instance, \cite{OsinAH_Groups}. 

\begin{prop}
    \label{prop: properties of groups acting acylindrically on trees}
    Let $\Gamma \curvearrowright T$ be a non-elementary acylindrical action of a group $\Gamma$ on a tree $T$. Then the following hold: 

    \begin{itemize}
        \item The centralizer $C_{\Gamma}(\gamma)$ of every hyperbolic isometry $\gamma \in \Gamma$ is cyclic (\cite[Corollary 6.9]{OsinAH_Groups}). 
        \item The centre of $\Gamma$ is finite (\cite[Corollary 7.2]{OsinAH_Groups}).
    \end{itemize}
\end{prop}

\subsection{Dehn twists}
\label{section: Dehn twists}

\begin{defn}
    \label{defn: Dehn twist}
    Let $G$ be a group. An automorphism $\delta \in \Aut(G)$ is called a \textbf{Dehn twist} of $G$ if either: 
    
    \begin{itemize}
        \item there exists a splitting of $G$ as an amalgam $G = A *_C B$ with $C$ abelian and an element $c_0 \in C$ such that $\delta \vert_A = id_A$ and $\delta(b) = c_0 b c_0^{-1}$ for each $b \in B$, or 
        \item there exists a splitting of $G$ as an HNN extension $G = A_{*D^e = C}$ with $D,C$ abelian and an element $w \in D$ such that $\delta \vert_A = id_A$ and $\delta(e) = ew$. 
    \end{itemize}
    
    In the first case, $\delta$ is called a \textbf{separating Dehn twist} of $G$ and in the second case, $\delta$ is called a \textbf{non-separating Dehn twist} of $G$.
    
\end{defn}

\begin{ex}
    \label{example: Geometric Dehn twist example}
    Let $\Sigma$ be an orientable surface and $\gamma \subset \Sigma$ a simple closed curve. Then the Dehn twist homeomorphism $T_{\gamma} \in \mathrm{Homeo}(\Sigma)$ about $\gamma$ induces a Dehn twist automorphism (in the sense of Definition \ref{defn: Dehn twist}) $\delta_{\gamma}$ of $\pi_1 (\Sigma)$. The Dehn twist automorphism $\delta_{\gamma}$ is separating (respectively, non-separating) if and only if the curve $\gamma$ is as well. 
\end{ex}

\section{Separating Dehn twists}
\label{section: separating Dehn twists}

We assume $G=G_1 * \cdots * G_k$ where $G_i \cong \mathbb Z^{n_i}$ for $n_i \in \mathbb{N}$.

Let $G = A *_C B$ where $A, B, C \leq G$. Let $\delta$ be a separating Dehn twist of this splitting, acting as follows for some $c_0 \in C$:  

$$ \delta : \begin{cases}
    a \to a & \forall a \in A \\
    b \to c_0bc_0^{-1} & \forall b \in B
\end{cases} $$

Throughout this section (and subsequent sections of the paper), we use the following classical theorem on the structure of subgroups of free products:

\begin{thm}[Kurosh's subgroup theorem]\cite[Theorem 3.01]{Kuhn}
\label{thm:Kurosh}

If $H$ is a subgroup of $G=*_{i \in I} G_i$ then there is a subset $X \subseteq G$, an index set $J$, subgroups $H_j$ of some $G_i$ for all $j \in J$ such that $$H=F(X) * (*_{j \in J} H_j^{\dag}).$$ 
where $H_j^{\dag}$ denotes a conjuguate of $H_j$. 
\end{thm}

We will also use the following result about amalgamated free products.

\begin{lemma}(\cite[Theorem 2.6]{Lyndon_Schupp})
    \label{lemma: normal form theorem for amalgamated free products}
    Let $G = A *_C B$. 
    
    \begin{enumerate}
        \item Each $g \in G$ can be written uniquely as: 

        $$g = c a_1 b_1 \cdots a_n b_n$$

        with $c \in C$ and $a_i$ in a choice of right coset transversal for $C$ in $A$ and $b_i$ in a choice of a right coset transversal of $C$ in $B$. 
        \item Suppose $n \geq 1$ and $a_1,\ldots,a_n \in A \setminus B$ and $b_1, \ldots, b_n \in B \setminus A$. Then $a_1 b_1 \cdots a_n b_n \neq 1$ in $G$. 
    \end{enumerate}

\end{lemma}

\begin{thm}
    \label{lemma: enlarging the splitting in separating case}
    Let $G$ be a finitely generated group that splits as a free product of free abelian groups. Suppose that $G = A_0 *_{C_0} B_0$ where $C_0$ is non-trivial abelian. Fix some $c_{0} \in C_0 \setminus 1$ and let $\delta$ be a separating Dehn twist of $G$ associated with this splitting and $c_0 \in C_0$. Then there exists a splitting $G = A *_C B$ with $A = \mathrm{Fix}(\delta)$ (the fixed point set of $\delta$), $C \geq C_0, B \geq B_0$, $C \leq A, B$ and $C$ maximal abelian in $G$ (and hence in $A$ and $B$) and such that $\delta$ still acts as a separating Dehn twist of this splitting. 
\end{thm}

\begin{proof}
    First, note that $C_0$ is maximal abelian in $A_0$ or $B_0$. Indeed, if not, then taking $a_0 \in A_0 \setminus C_0$ and $b_0 \in B_0 \setminus C_0$ centralizing $C_0$, since $C_0 \neq 1$ and $G$ is commutative transitive (as it is CSA), we have that $a,b$ commute. However, this contradicts the normal form theorem, since there can be no relation between elements of $A_0 \setminus C_0$ and elements of $B_0 \setminus C_0$. 

    First, we will show that we can enlarge the splitting to a splitting $G = A *_C B$ where $C$ is malnormal in $A$ and $B$, and then we will show that in fact $A = \mathrm{Fix}(\delta)$. 

    Without loss of generality, suppose that $C_0$ is maximal abelian in $B_0$ and let $C_A$ be the centralizer of $C_0$ in $A_0$ (which is the maximal abelian subgroup of $A_0$ containing $C_0$). We can assume that $C_A \neq C_0$, since otherwise $C_0$ is already malnormal in both $A_0$ and $B_0$ and we are done. Let $A = A_0, B = \langle B_0, C_A \rangle$ and $C = C_A$. We claim that $G = A *_C B$, that $C$ is maximal abelian in both $A$ and $B$ and $\delta$ still acts as a separating Dehn twist of this splitting with respect to $c_0$. 

\vs

\underline{$C$ is maximal abelian in $A$}: Since $C$ is the centralizer of $C_0$ in $A$, we have that $C$ is maximal abelian in $A$.

\vs

\underline{$C$ is maximal abelian in $B$}: Suppose $g \in B$ is in the maximal abelian subgroup of $B$ containing $C$. Then $g$ commutes with some $c \in C \setminus C_0$. We can write $g = b_1 \cdots a_n$ for some $a_i \in C_A \setminus B_0 \subset A_0 \setminus B_0 = A_0 \setminus C_0$ and $b_i \in B_0 \setminus C_A \subset B_0 \setminus C_0$, with all letters non-trivial except possibly $b_1$ or $a_n$. We have $cgc^{-1}g^{-1} = 1$, which yields: 

$$ c (b_1 a_1 \cdots a_n) c^{-1} (b_1 a_1 \cdots b_n a_n)^{-1} = 1 \iff c b_1 a_1 \cdots b_n a_n c^{-1}a_n^{-1} b_n^{-1} \cdots a_1^{-1} b_1^{-1} = 1$$

Since $a_n \in C$, this becomes:

$$ c b_1 a_1 \cdots b_n c^{-1} b_n^{-1} \cdots a_1^{-1} b_1^{-1} = 1$$

Since $c \in C \setminus C_0$, we have $c \in A_0 \setminus C_0$. If $n \geq 1$ and $g \notin C$, then $b_n \notin C_A$, so $b_n \in B_0 \setminus C_0$, We thus have an alternating word in $A_0 \setminus B_0$ and $B_0 \setminus C_0$ being the trivial element, contradicting the normal form theorem for $G = A_0 *_{C_0} B_0$. Therefore, $n = 1$ and $g = a_1 \in C_A$. We conclude that $C$ is maximal abelian, and hence malnormal, in $B$.

\vs

\underline{$A \cap B = C$}: Note that $C \subseteq A,B$ and hence $C \subseteq A \cap B$. Let $g \in A \cap B$. Since $g \in B = \langle B_0, C_A \rangle$, we can write $g = a_1 b_1 \cdots a_n b_n$ for some $a_i \in C_A \setminus B_0 \subset A_0 \setminus B_0 = A_0 \setminus C_0$ and $b_i \in B_0 \setminus C_A \subset B_0 \setminus C_0$. Suppose for contradiction that $g \notin C$, so there is at least one letter $b_i$. We then have: 

$$1 = a_1 b_1 \cdots a_n b_n g^{-1}$$

Since $g \in A = A_0$, this is an alternating non-trivial word in $A_0 \setminus C_0$ and $B_0 \setminus C_0$, contradicting the normal form theorem for $G = A_0 *_{C_0} B_0$. Therefore, we must have $g \in C_A = C$. 

\vs

\underline{$G = A *_C B$}: Since $G = \langle A_0, B_0 \rangle$ and since $A \geq A_0, B \geq B_0$ we have $G = \langle A, B\rangle$. Since we showed above that $A \cap B = C$, it remains to show the normal form theorem for $A,B,C$. Suppose that: 

$$c a_1 b_1 \cdots a_n b_n = 1$$

for $c \in C, a_i \in A \setminus C, b_i \in B \setminus C$.

Note that $A \setminus C \subset A_0 \setminus C_0$, so each $a_i \in A_0 \setminus C_0$. 

We can write each $b_i$ as $b_i = b_{i1}a_{i1} \cdots b_{im_i} a_{im_i}$ for $b_{ij} \in B_0 \setminus C_A$ and $a_{ij} \in C_A \setminus B_0$.  Substituting these expressions above, we obtain: 

$$ c a_1 (b_{11}a_{11} \cdots b_{1m_1} a_{1m_1})a_2 \cdots a_n (b_{n1}a_{n1} \cdots b_{nm_n} a_{nm_n}) = 1 $$

Note that $B_0 \setminus C_A \subset B_0 \setminus C_0$ and $C_A \setminus B_0 \subset A_0 \setminus C_0$, so each $b_{ij} \in B_0 \setminus C_0$ and $a_{ij} \in A_0 \setminus C_0$. Also, note that $ca_1 \notin C_0$, since otherwise $a_1 \in C$. Thus, we obtain a non-empty word in $A_0 \setminus C_0$ and $B_0 \setminus C_0$ equal to the identity in $G$, contradicting the normal form theorem for $G = A_0 *_{C_0} B_0$. We conclude that the word the $c a_1 b_1 \cdots a_n b_n$ must be empty, and so the normal form theorem holds for $A,B,C$ and thus that $G = A *_C B$. 

\vs 

\underline{$\delta$ still acts as a separating Dehn twist of $G = A *_C B$}: We still have $\delta$ being the identity on $A = A_0$. We need to show that $\delta$ acts by conjugation via $c_0$ on $B$. We know this is true on $B_0$. On $C_A$, $\delta$ fixes $C_A$ pointwise since $C_A \subset A$. Also, since $C_A$ is abelian and contains $c_0$, we have that $\delta$ acts as conjugation by $c_0$ on $C_A$ (which is trivial). Therefore, $\delta$ acts as conjugation by $c_0$ on $B$ as well. We conclude that $\delta$ acts as a separating Dehn twist of $G = A *_C B$. 

\vs 

Now we show that $A = \text{Fix}(\delta)$. Indeed, by definition $A \subseteq \text{Fix}(\delta)$. Let $h \in \text{Fix}(\delta)$. As $h \in A_0 *_{C_0} B_0$, it has a normal form, say, $h = c_h a_1b_1 \cdots a_nb_n$ where $c_h \in C_0$, each $a_i \in A_0 \setminus C_0$ and each $b_i \in B_0 \setminus C_0$. Now 
    \[
    c_h a_1(b_1)^{c_0} \cdots a_n(b_n)^{c_0} = \delta (h) = h = c_h a_1b_1 \cdots a_nb_n
    \implies (b_1)^{c_0} \cdots a_n(b_n)^{c_0} b_n^{-1}a_n^{-1} \cdots b_1^{-1} = 1
    \]

Then there must be cancellations, in particular $(b_n)^{c_0} b_n^{-1} \in C_0$ (otherwise it will contradict the fact that every reduced alternating word in $A_0 *_{C_0} B_0$ is non-trivial). This implies that $b_n c_0^{-1} b_n^{-1} \in C_0 \implies C_0^{b_n} \cap C_0\neq 1$. But $C_0$ is maximal abelian in $B_0$, hence malnormal in $B_0$. Thus, we obtain $b_n \in C_0$. Then $(b_n)^{c_0} b_n^{-1} = 1$ will reduce $(b_1)^{c_0} \cdots a_n(b_n)^{c_0} b_n^{-1}a_n^{-1} \cdots b_1^{-1} = 1$ to $(b_1)^{c_0} \cdots a_{n-1}(b_{n-1})^{c_0} b_{n-1}^{-1}a_{n-1}^{-1} \cdots b_1^{-1} = 1$. By the same argument as before, we can show that $b_{n-1} \in C_0$. Continuing the same way we get each $b_i$ is in $C_0$ implies hence $h \in A$. Hence $\text{Fix} (\delta) = A$.

\vs

We will next show that $C$ is maximal abelian in $G$. Let $C_G$ be the maximal abelian subgroup of $G$ containing $C$. We will show that $C_G \subseteq C$. Let $g \in C_G$. Then $g$ has a unique normal form, say, $g = c' a_1b_1 \cdots a_n b_n$, for $c' \in C_0$, $a_i \in A_0 \setminus C_0$ and $b_i \in B_0 \setminus C_0$. Since $g \in C_G$, we have that $g$ commutes with $c_0$, and hence $c_0 g c_0^{-1} = g$. This yields:

$$ c' a_1^{c_0}b_1^{c_0} \cdots a_n^{c_0} b_n^{c_0} = c' a_1b_1 \cdots a_n b_n \implies a_1b_1 \cdots a_n b_n (b_n^{-1})^{c_0}(a_n^{-1})^{c_0} \cdots (b_1^{-1})^{c_0} (a_1^{-1})^{c_0} = 1$$

By using the same method for proving $\text{Fix} (\delta) = A$ above (which boiled down to centralizers of $C_0$ being maximal abelian and hence malnormal), we can deduce that each $a_i$ must be in the centralizer $C_A$ of $C_0$ in $A_0$ and each $b_i$ must be in the centralizer of $C_0$ in $B_0$, which is $C_0$ since $C_0$ is maximal abelian in $B_0$. This implies that $g \in A_0$. This proves the claim.

Therefore, $\text{Fix}(\delta) = A_0 = A$ and $C = C_A = C_G$ is maximal abelian in $G$ and hence maximal abelian in both $A$ and $B$.

\end{proof}

We will assume from now on that $C$ is maximal abelian in $G$, and that $A$ and $B$ are non-abelian so that the splitting is non-trivial. We will need the following theorem, which follows from \cite[Theorem 1.3]{Finiteness_properties_graphs_of_groups} since $C$ is finitely generated, as it is either cyclic or contained in an abelian free factor of $G$. 

\begin{lemma}
    \label{lemma: A and B are finitely generated in separating case}
    Let $G$ be a free product of free abelian groups. Assume that $G$ is finitely generated. If $G$ splits as an amalgam 

    $$G = A *_C B$$ with $C$ abelian, then $A$ and $B$ are finitely generated.
\end{lemma}

\subsection{Step I: Finding a canonical splitting of $\Gamma_{\delta}$}
\label{subsection: Step 1}

We first find a splitting for $\Gamma_\delta$. We let $A=\langle S_A \vert R_1 \rangle$, $B=\langle S_B \vert R_B \rangle$ and $i_A:C \to A$ and $i_B: C\to B$ be the embeddings of the edge group $C$ into $A$ and $B$ respectively.

By definition, we have that 
\begin{align*}
        \Gamma_\delta &= \langle A,B,t \, \vert \, R_A, R_B, tat^{-1}=a, tbt^{-1}=c_0bc_0^{-1}, i_A(c)=i_B(c) \quad \forall a\in A, \forall b \in B, \forall c\in C \rangle \\
        &= \langle A,B,t \, \vert \, R_{A}, R_{B}, tat^{-1}=a, (c_{0}^{-1}t)b(c_{0}^{-1}t)^{-1}=b, i_{A}(c)=i_{B}(c) \quad \forall a \in A, \forall b\in B, \forall c\in C \rangle \\
        &=(A \oplus \langle t\rangle) *_{C \oplus \langle t \rangle =C \oplus \langle t_1 \rangle} (B \oplus \langle t_1 \rangle) \\
        &=(A \oplus \langle t\rangle) *_{\mathbb Z^{m+1}} (B \oplus \langle t_1 \rangle)
\end{align*}

with $t_1=c_0^{-1}t$ and $m$ the rank of the abelian group $C$. We denote $\tilde{A} = A \oplus \langle t \rangle$, $\tilde{B} = B \oplus \langle t_1 \rangle$ and $\tilde{C} = C \oplus \langle t \rangle = C \oplus \langle t_1 \rangle$. Then by above we have the following splitting of $\Gamma_{\delta}$:

$$\Gamma_{\delta} = \tilde{A} *_{\tilde{C}} \tilde{B}$$

We let $T$ be the Bass-Serre tree of $\Gamma_\delta$ with respect to this splitting. 

\begin{lemma}
    \label{lemma: acylindricity of separating mapping torus}
    The action of $\Gamma_\delta$ on $T$ is $2$-acylindrical.
\end{lemma}

\begin{proof}
    Consider a path $[gv, u]\cup[u, v]\cup[v, hu] = ge \cup e \cup he$ for any $g \in \Gamma_u\setminus\Gamma_e$ and $h \in \Gamma_v\setminus\Gamma_e$, where $e = [u, v] = [\tilde{A},\tilde{B}]$ is the fundamental domain in the Bass-Serre tree $T$. We will prove that the three edge stabilizers intersect trivially. Note that this will be sufficient to prove the $2-$acylindricity of the action of $\Gamma_\delta$ on $T$ since any geodesic segment of length 3 is a $\Gamma$-translate of the segment $ge \cup e \cup he$. $\mathrm{Stab}(ge) \cap \mathrm{Stab}(e) = g\mathrm{Stab}(e)g^{-1} \cap \mathrm{Stab}(e) = g(C \oplus \langle t \rangle)g^{-1} \cap (C \oplus \langle t \rangle) = (gCg^{-1} \cap C) \oplus \langle t \rangle = \langle t \rangle$, because $C$ is malnormal in $A$ and $B$. Similarly, $\mathrm{Stab}(he) \cap \mathrm{Stab}(e) = \langle t_1 \rangle$. As $t_1 = c_0t$, $\langle t_1 \rangle \cap \langle t \rangle = 1$.
\end{proof}

\begin{lemma}
    \label{lem: Canonicity lemma for separating Dehn twists}
    Let $T$ be the Bass--Serre tree of the above splitting of $\Gamma$. Let $S$ be a minimal $\Gamma$-tree satisfying the following: 

    \begin{enumerate}
        \item Vertex stabilizers of $S$ are of the form $H \rtimes \Z$ where $H < G$ is non-trivial and where the semi-direct product is a direct product if $H$ is non-abelian. 
        \item Edge stabilizers of $S$ are abelian.
        \item The action of $\Gamma$ on $S$ is 2-acylindrical.
    \end{enumerate}
    Then $T$ and $S$ are $\Gamma$-equivariantly isomorphic.
\end{lemma}

\begin{proof}
    We will first construct a $\Gamma$-equivariant injection between the vertex sets $V(T)$ and $V(S)$ of $T$ and $S$. 

    We first show that $\tilde{A}$ is elliptic on $S$. 
    
    \begin{lemma}
        \label{lem: tilde A elliptic on S separating}
        Let $T$ and $S$ be as in the statement of Lemma \ref{lem: Canonicity lemma for separating Dehn twists}. Then $\tilde{A}$ and $\tilde{B}$ fix unique respective vertices $v_{\tilde{A}}$ and $v_{\tilde{B}}$ of $S$. 
    \end{lemma}

    \begin{proof}[Proof of Lemma \ref{lem: tilde A elliptic on S separating}]
        Suppose for contradiction that $\tilde{A}$ had some $g \in \tilde{A}$ hyperbolic on $S$. We will show that the centralizer $C_{\Gamma}(g)$ of $g$ in $\Gamma$ contains $\Z^2$. Indeed, if $g \notin \langle t \rangle$, then $\langle g, t \rangle < C_{\Gamma}(g)$ and $\langle g, t \rangle \cong \Z^2$, and if $g \in \langle t \rangle$, then taking any non-trivial $a \in A$ yields $\langle a, g \rangle < C_{\Gamma}(g)$ and $\langle a, g \rangle \cong \Z^2$. Therefore, $C_{\Gamma}(g)$ contains a $\Z^2$ subgroup. This contradicts Proposition \ref{prop: properties of groups acting acylindrically on trees} (1), since by minimality of $\Gamma \curvearrowright S$ and the fact that $\Gamma$ is not virtually cyclic, we have that $\Gamma \curvearrowright S$ is non-elementary acylindrical. 

    Thus, $\tilde{A}$ fixes a vertex $v_{\tilde{A}}$ of $S$. This vertex is unique since otherwise $\tilde{A}$ fixes an edge of $S$, which implies that $\tilde{A}$ is abelian, a contradiction.

    Similarly, $\tilde{B}$ fixes a unique vertex $v_{\tilde{B}}$ of $S$. 
    \end{proof}

By Lemma \ref{lem: tilde A elliptic on S separating}, it follows that for each $g \in \Gamma$, the stabilizer of the vertex $g \tilde{A}$ (respectively, $g \tilde{B}$) fixes the unique vertex $g v_{\tilde{A}}$ (respectively, $g v_{\tilde{B}}$) of $S$. This yields a $\Gamma$-equivariant map $f: V(T) \to V(S)$ defined by a sending a vertex $v$ of $T$ to the unique of $S$ stabilized by $\mathrm{Stab}_{\Gamma}(v)$

    We now show that $f$ is injective. It suffices to show that the stabilizer $\Gamma_{v_{\tilde{A}}}$ of $v_{\tilde{A}}$ has unique fixed vertex $\tilde{A}$. It will follow similarly that $\Gamma_{v_{\tilde{B}}}$ fixes the unique vertex $\tilde{B}$ of $T$. 
    
    We have that the stabilizer $\Gamma_{v_{\tilde{A}}} \cong H \rtimes \Z$ with $H$ non-abelian. Indeed, we have $\tilde{A} < \Gamma_{v_{\tilde{A}}}$ by Lemma \ref{lem: tilde A elliptic on S separating}, and $A$ is a free product of free abelian groups with at least 2 free factors, hence $A$ contains the free group $F_2$, and thus so does $\tilde{A}$. Therefore, $H$ cannot be abelian, since otherwise $H \rtimes \Z$ would be solvable and contain $F_2$, a contradiction. Since $H$ is non-abelian, we have that $\Gamma_{v_{\tilde{A}}} \cong H \oplus \Z$, and so arguing as in Lemma \ref{lem: tilde A elliptic on S separating}, we obtain that $\Gamma_{v_{\tilde{A}}}$ fixes a unique vertex of $T$. This vertex must be $\tilde{A}$, since otherwise $\tilde{A}$ would fix more than one vertex, thus it would fix an edge of $T$, and hence would be abelian. Similarly, it follows that $gv_{\tilde{A}}$ fixes the unique vertex $g\tilde{A}$ of $T$, and similarly for $g v_{\tilde{B}}$. This gives a $\Gamma$-equivariant left inverse $h$ of $f$ defined by $h(gv_{\tilde{A}}) = \tilde{A}$ and $h(gv_{\tilde{B}}) = \tilde{B}$. Therefore, $f$ is injective, and hence a $\Gamma$-equivariant bijection onto its image. 
    
    We now show that $f$ extends to a map of trees. To do this, we need to show that $[v_{\tilde{A}}, v_{\tilde{B}}]$ is an edge of $S$.
    \begin{lemma}
        \label{lem: map on vertices is tree map}
        Let $f: V(T) \to V(S)$ be the map obtained from Lemma \ref{lem: tilde A elliptic on S separating}. Then $[f(\tilde{A}), f(\tilde{B})] = [v_{\tilde{A}}, v_{\tilde{B}}]$ is an edge in $S$.  
    \end{lemma}

    \begin{proof}
        To show that $[v_{\tilde{A}}, v_{\tilde{B}}]$ is an edge in $S$, we need to show that $d_S(v_{\tilde{A}}, v_{\tilde{B}}) = 1$. We cannot have $d_S(v_{\tilde{A}}, v_{\tilde{B}}) = 0$ since $f$ is injective on $V(T)$ and we cannot have $d_S(v_{\tilde{A}}, v_{\tilde{B}}) > 2$ by 2-acylindricity of $S$. Thus, it remains to rule out the case $d_S(v_{\tilde{A}}, v_{\tilde{B}}) = 2$. 

    If $d_S(v_{\tilde{A}}, v_{\tilde{B}}) = 2$, let $x$ be the middle vertex of the segment $[v_{\tilde{A}}, v_{\tilde{B}}]$. Since $x \neq v_{\tilde{A}}$ and $v_{\tilde{A}}$ is the only vertex of $S$ that $\tilde{A}$ fixes, we have that $\tilde{A}$ does not fix $x$, so there exists $g \in \tilde{A}$ such that $gx \neq x$. Then the segment $[v_{\tilde{B}}, gv_{\tilde{B}}]$ of $S$ has length at least 4, and hence has trivial pointwise stabilizer by 2-acylindricity of $S$. However, by $\Gamma$-equivariance of $f$, for each $g \in \tilde{A}$, we have 

    $$\mathrm{Stab}([v_{\tilde{B}}, gv_{\tilde{B}}]) = \Gamma_{v_{\tilde{B}}} \cap \Gamma_{v_{\tilde{B}}}^g = \tilde{B} \cap \tilde{B}^g = \tilde{C} \cap \tilde{C}^g$$

    Since $C$ is malnormal in $A$, we have that $\tilde{C} \cap \tilde{C}^g = \langle t \rangle \neq 1$. Thus, we have a contradiction. We conclude that $[v_{\tilde{A}}, v_{\tilde{B}}]$ is an edge of $S$, i.e.\ that $d_S(v_{\tilde{A}}, v_{\tilde{B}}) = 1$. 
    \end{proof}
    
    Hence, $f$ is an isomorphism of $\Gamma$-trees onto its image $f(T)$. But since $S$ is minimal, it follows that $f(T) = S$. Hence, $f$ is a $\Gamma$-equivariant isomorphism from $T$ to $S$.

\end{proof}

\subsection{Step II: Finding a fibre and splitting preserving isomorphism.}
The following result is well-known but we record its proof for the convenience of the reader. 

\begin{lemma}
\label{lem:mapping_tori_iso_implies_conjuguate}
Let $G$ be any group and $\psi, \phi$ be automorphisms of $G$. Then $\psi$ and $\phi$ are conjugate in $\Aut(G)$ if and only if there is an isomorphism $f: G \rtimes_{\phi} \langle t \rangle \to G \rtimes_{\psi} \langle s \rangle$ such that $f(G) \subseteq G$ and $f(t)=s$.
\end{lemma}

\begin{proof}
    Suppose first that there exists such an isomorphism $f$. Write $F=f|_{G}\in \Aut(G)$, then for all $\gamma \in G$,
    \[
    F^{-1} \circ \phi \circ F (\gamma) = F^{-1}(\phi(F(\gamma)) =f^{-1}(\phi(F(\gamma))=
    f^{-1}(tF(\gamma)t^{-1})= f^{-1}(t) \gamma f(t^{-1}) =s\gamma s^{-1} = \psi(\gamma)
    \]

    Now suppose that $\psi$ and $\phi$ are conjugate in $\Aut(G)$ by an automorphism $F \in \Aut(G)$, i.e.\ that $\psi = F \phi F^{-1}$. Define a map $f: G \rtimes_{\phi} \langle t \rangle \to G \rtimes_{\psi} \langle s \rangle$ by putting $f \vert_G = F$ and $f(t) = s$. We claim that $f$ extends to an isomorphism $G \rtimes_{\phi} \langle t \rangle \to G \rtimes_{\psi} \langle s \rangle$. Indeed, to first check that $f$ is a homomorphism we need to check that the relation 

    $$tgt^{-1} = \phi(g)$$

    is preserved by $f$, i.e.\ that $f(t)f(g)f(t)^{-1} = f(\phi(g))$. We have: 

    \begin{align*}
        f(t)f(g)f(t)^{-1} &= s F(g) s^{-1} \\
        &= \psi(F(g)) \\
        &= F(\phi(g)) \text{, since $\psi = F\phi F^{-1}$}\\
        &= f(\phi(g))
    \end{align*}
    Therefore, $f$ extends to a homomorphism. 

    We have that $f$ is invertible since $f^{-1}$ is given by $f^{-1} \vert_G = F^{-1}$ and $f^{-1}(s) = t$ (that $f^{-1}$ is a homomorphism follows exactly as the argument to show that $f$ is a homomorphism). Therefore, $f$ is an isomorphism such that $f(G) \subseteq G$ and $f(t)=s$.
\end{proof}

We record another algebraic lemma that we will need. 

\begin{lemma}
    \label{lem:normalizer of A tilde}
    Let $G = A *_C B$ be a maximal splitting of $G$ with respect to some Dehn twist $\delta$. Let $\Gamma = G \rtimes_{\delta} \langle t \rangle$. If $\gamma \in \Gamma$ is such that $\tilde{A} \subseteq \tilde{A}^{\gamma}$, then $\gamma \in \tilde{A}$. Similarly if $\gamma \in \Gamma$ is such that $\tilde{B} \subseteq \tilde{B}^{\gamma}$, then $\gamma \in \tilde{B}$. In particular, the normalizers in $\Gamma$, $N_{\Gamma}(\tilde{A}) = \{\gamma \in \Gamma \, | \, \tilde{A}^{\gamma} = \tilde{A}\} =  \tilde{A}$ and $N_{\Gamma}(\tilde{B}) = \{\gamma \in \Gamma \, | \, \tilde{B}^{\gamma} = \tilde{B}\} = \tilde{B}$. 
\end{lemma}

\begin{proof}
    
    Let $\gamma$ be such that $\tilde{A} \subseteq \tilde{A}^{\gamma}$. Then $\tilde{A}$ fixes the vertices $\tilde{A}$ and $\gamma \tilde{A}$ in the Bass--Serre tree $T$ of the induced splitting of $\Gamma$. If $\gamma \notin \tilde{A}$, then $\gamma \tilde{A} \neq \tilde{A}$, and so $d_T(\tilde{A}, \gamma \tilde{A}) \geq 2$. Then the geodesic path in $T$ from $\tilde{A}$ to $\gamma \tilde{A}$ contains an edge of $T$, which implies that $\mathrm{Stab}_{\Gamma}([\tilde{A}, \gamma \tilde{A}])$ is abelian (since edge stabilizers of $T$ are abelian). However, $\mathrm{Stab}_{\Gamma}([\tilde{A}, \gamma \tilde{A}]) = \tilde{A} \cap \tilde{A}^{\gamma} = \tilde{A}$. Therefore, we obtain that $\tilde{A}$ is abelian, a contradiction. 
   Therefore, we must have that $\gamma \in \tilde{A}$. 

   The proof for $\tilde{B}$ is analogous, since $\tilde{B}$ is non-abelian. 
\end{proof}

\begin{lemma}
\label{lem:upgrade_to_fiber_preserving}
    Let $\delta, \delta' \in \Aut(G)$ be separating Dehn twists. As in Section \ref{subsection: Step 1}, decompose their mapping tori as $\Gamma_{\delta} =(A \oplus \langle t\rangle) *_{C \oplus \langle t \rangle} (B \oplus \langle t_1 \rangle)$ and $\Gamma_{\delta'} =(A' \oplus \langle t'\rangle) *_{C' \oplus \langle t' \rangle} (B' \oplus \langle t_1' \rangle)$. Then $\delta$ and $\delta'$ are conjugate in $\Aut(G)$ if and only if there are isomorphisms $\phi_A : A \to A'$ and $\phi_B : B \to B'$ such that: 
    
    \begin{enumerate}
        \item $\phi_A(C) = C'$ 
        \item $\phi_B(C) = C'$ 
        \item $\phi_A \vert_C = \phi_B \vert_C$ 
        \item $\phi_A(c_0) = \phi_B(c_0) = c_0'$
    \end{enumerate}

\end{lemma}

\begin{proof}
    As in \ref{subsection: Step 1}, denote $\tilde{A} = A \oplus \langle t \rangle$, $\tilde{B} = B \oplus \langle t_1 \rangle$, $\tilde{C} = C \oplus \langle t \rangle = C \oplus \langle t_1 \rangle$, etc. 
    
    Suppose first that $\delta$ and $\delta'$ are conjugate in $\Aut(G)$. Then by Lemma \ref{lem:mapping_tori_iso_implies_conjuguate}, since $\delta, \delta'$ are conjugate in $\Aut(G)$, there exists an isomorphism $f: \Gamma_{\delta} \to \Gamma_{\delta'}$ inducing an automorphism of $G$ and mapping $t$ to $t'$. Let $T$ and $T'$ be the Bass-Serre trees corresponding to the given splittings of $\Gamma_{\delta}$ and $\Gamma_{\delta'}$.

     The isomorphism $f$ induces an action $\cdot_f$ of $\Gamma_{\delta}$ on $T'$ via, for all $g \in \Gamma_{\delta}$ and all $x \in T'$, $g \cdot_f x = f(g) \cdot x$, where the $\cdot$ on the right hand side is the action of $\Gamma_{\delta'}$ on $T'$. Since $T'$ is 2-acylindrical with respect to the $\Gamma_{\delta'}$-action, it is 2-acylindrical with respect to the $\Gamma_{\delta}$-action. The vertex stabilizers of $T'$ with respect to the $\Gamma_{\delta}$ action are isomorphic to $f^{-1}(A' \oplus \Z) \cong f^{-1}(A') \oplus \Z$ or $f^{-1}(B' \oplus \Z) \cong f^{-1}(B') \oplus \Z$. Since $f$ is an isomorphism preserving $G$, we have that $f^{-1}(A') < G$ and $f^{-1}(B') < G$ and both $f^{-1}(A')$ and $f^{-1}(B')$ are non-abelian since $A',B'$ are non-abelian. Also, the edge stabilizers of $T'$ with respect to the $\Gamma_{\delta}$ action are of isomorphic to $f^{-1}(C' \oplus \Z)$, which is abelian. Therefore, $T'$ satisfies the conditions of Lemma \ref{lem: Canonicity lemma for separating Dehn twists} as a $\Gamma_{\delta}$-tree and hence $T$ and $T'$ are isomorphic as $\Gamma_{\delta}$-trees. 
     
     By Lemma \ref{lem: tilde A elliptic on S separating}, we have that $\tilde{A}$ fixes a unique vertex $\gamma \tilde{A}'$ of $T'$ for some $\gamma \in \Gamma'$. We can assume that $\gamma \in G$ since we can absorb the $t'$-component of $\gamma$ into $\tilde{A}'$. We then have $f(\tilde{A}) \subseteq \mathrm{Stab}_{\Gamma'}(\gamma \tilde{A}') = (\tilde{A}')^{\gamma}$. Similarly, we have $f^{-1}(\tilde{A}') \subseteq \tilde{A}^{\lambda}$ for some $\lambda \in G$.
     
     We will show that $f(\tilde{A}) = (\tilde{A}')^{\gamma}$. From the inclusions $f(\tilde{A}) \subseteq (\tilde{A}')^{\gamma}$ and $f^{-1}(\tilde{A}') \subseteq \tilde{A}^{\lambda}$, we have $f(\tilde{A}) \subseteq (\tilde{A}')^{\gamma} \subseteq f(\tilde{A}^{\lambda})^{\gamma} = f(\tilde{A}^{f^{-1}(\gamma)\lambda})$. Hence, we obtain $\tilde{A} \subseteq \tilde{A}^{f^{-1}(\gamma) \lambda }$, which implies by Lemma \ref{lem:normalizer of A tilde} that $f^{-1}(\gamma) \lambda \in \tilde{A}$ and hence that $\tilde{A}^{f^{-1}(\gamma) \lambda } = \tilde{A}$. Therefore, we have $f(\tilde{A}) \subseteq (\tilde{A}')^{\gamma} \subseteq f(\tilde{A}^{f^{-1}(\gamma)\lambda}) = f(\tilde{A})$. Thus, $f(\tilde{A}) = (\tilde{A}')^{\gamma}$.

     We similarly obtain that $f(\tilde{B)} = (\tilde{B}')^{\eta}$ for some $\eta \in G$. By Lemma \ref{lem: tilde A elliptic on S separating}, we obtain that $[\gamma \tilde{A}', \eta \tilde{B}']$ is an edge in $T'$ (because the isomorphism $T \to T'$ maps $\tilde{A}$ to $\gamma \tilde{A}$ and $\tilde{B}$ to $\eta \tilde{B}$, and $[\tilde{A}, \tilde{B}]$ is an edge in $T$, hence $[\gamma \tilde{A}', \eta \tilde{B}']$ is an edge in $T'$). Thus, we have $\eta \tilde{B}' = \alpha \gamma \tilde{B}'$ for some $\alpha \in \mathrm{Stab}_{\Gamma'}(\gamma \tilde{A}') = (\tilde{A}')^{\gamma}$, so we can assume $\eta = \alpha\gamma$, and in fact that $\eta = \gamma$ by replacing $\gamma$ with $\alpha^{-1} \gamma$, since $(\tilde{A'})^{\alpha \gamma} = (\tilde{A}')^{\gamma}$. We then have $f(\tilde{C}) = f(\tilde{A} \cap \tilde{B}) = f(\tilde{A}) \cap f(\tilde{B}) = (\tilde{A}')^{\gamma} \cap (\tilde{B}')^{\gamma} = (\tilde{C}')^{\gamma}$.
     
     Let $\phi = \mathrm{Ad}(\gamma^{-1})f$. Since $f \vert_{\tilde{A}}$ is an isomorphism from $\tilde{A}$ to $(\tilde{A}')^{\gamma}$, we have that $\phi \vert_{\tilde{A}}$ is an isomorphism from $\tilde{A}$ to $\tilde{A}'$. Similarly, $\phi \vert_{\tilde{B}}$ is an isomorphism from $\tilde{B}$ to $\tilde{B}'$ and $\phi \vert_{\tilde{A}}(\tilde{C}) = \tilde{C}' = \phi \vert_{\tilde{B}}(\tilde{C})$. Furthermore, $\phi \vert_{\tilde{A}}, \phi \vert_{\tilde{B}}$ agree on $\tilde{C}$ since they are restrictions of the same map $\phi$. Next, since $\phi \vert_{\tilde{A}}$ is an isomorphism from $\tilde{A}$ to $\tilde{A}'$ and since $\langle t \rangle = Z(\tilde{A})$ (respectively, $\langle t' \rangle = Z(\tilde{A}')$), we can assume up to changing the sign of $\phi$ on $\langle t \rangle$ that $\phi(t) = t'$. Therefore, $\phi$ is an isomorphism $\Gamma \to \Gamma'$ satisfying $\phi(G) = G$ and $\phi(t) = t'$, which yields that $\delta' = \delta^{\phi}$. We then have $\phi(A) \subseteq \tilde{A}' \cap G = A'$ and similarly $\phi^{-1}(A') \subseteq \tilde{A} \cap G = A$, so $\phi(A) = A'$. Similarly, $\phi(B) = B'$ and $\phi(C)=C'$.
     
     We now claim that $\phi(c_0) = c_0'$. Indeed, since $\delta' = \delta^{\phi}$, we have that $\delta'$ acts on $\phi(B) = B'$ as conjugation by $\phi(c_0)$. Also, by definition, $\delta'$ acts on $B'$ as conjugation by $c_0'$. Thus, $(c_0')^{-1}\phi(c_0) \in C_G(B') = 1$ since $B'$ is a non-abelian subgroup of $G$. Thus, we obtain $\phi(c_0) = c_0'$.

     Now let $\phi_A = \phi \vert_A, \phi_B = \phi \vert_B$. We then have that $\phi_A : A \to A'$ and $\phi_B : B \to B'$ are isomorphisms that map $C$ to $C'$, $c_0$ to $c_0'$ and agree on $C$, and hence are the required isomorphisms.

     \vspace{0.2in}
     
    Conversely, suppose there exist such an isomorphisms $\phi_A,\phi_B$. Since $\phi_A, \phi_B$ agree on $C$ and map $C$ to $C'$, the free product isomorphism $\phi_A * \phi_B$ descends to an isomorphism $\Phi : G \to G$ preserving the splittings $G = A *_C B = A' *_{C'} B'$. Since $\phi_A(c_0) = \phi_B(c_0) = c_0'$, mapping $t$ to $t'$ extends $\Phi$ to an isomorphism $\tilde{\Phi} : \Gamma_{\delta} \to \Gamma_{\delta'}$. By Lemma \ref{lem:mapping_tori_iso_implies_conjuguate}, we have that $\delta$ and $\delta'$ are conjugate in $\Aut(G)$.

\end{proof}

\begin{rmk}
    \label{remark: simplified conjugacy conditions for separating}
    In Lemma \ref{lem:upgrade_to_fiber_preserving}, the four conditions on $\phi_A$ are equivalent to the last two, since by Theorem \ref{lemma: enlarging the splitting in separating case}, $C$ is the maximal abelian subgroup of $G$ containing $c_0$ and hence is the centralizer of $c_0$ in $G$ (and thus in $A$ and $B$). Thus, conditions (1) and (2) follow from (3) and (4) since $\phi_A, \phi_B$ are homomorphisms.
\end{rmk}

\subsection{Step III: An algorithm to decide if $\delta$ and $\delta'$ are conjugate.}\label{section: algorithm for separating Dehn twists}

\begin{lemma}
\label{lem:Tietze_algo}
    Given any separating Dehn twist $\delta \in \Aut(G)$ and a starting presentation $G = \langle S \vert R \rangle$ of $G$, there is an algorithm to compute the splitting $$\Gamma_{\delta} = (A \oplus \langle t \rangle) *_{C \oplus \langle t \rangle = C \oplus \langle t_1 \rangle} (B \oplus \langle t_1 \rangle)$$
\end{lemma}

\begin{proof}

As shown in the beginning of Section \ref{subsection: Step 1}, given that $\delta$ is a Dehn twist, there exists a splitting of its mapping torus as an amalgamated free product over a free abelian subgroup: 

$$\Gamma_{\delta} = (A \oplus \langle t \rangle) *_{C \oplus \langle t \rangle = C \oplus \langle t_1 \rangle} (B \oplus \langle t_1 \rangle)$$

Thus, starting from the standard presentation $\Gamma_{\delta} = \langle S, t \vert R, tst^{-1} = \delta(s), \text{ for each } s \in S \rangle$, there is a finite sequence of Tietze transformations that will take this given presentation to the presentation associated with the above splitting of $\Gamma_{\delta}$ as an amalgamated free product. 

Thus, we obtain an algorithm to compute the splitting as follows. We enumerate all presentations of $\Gamma_{\delta}$ obtained from the standard presentation by applying finitely many Tietze transformation, checking at each step whether the current presentation is equal to the presentation associated with the above splitting of $\Gamma_{\delta}$. Since $\Gamma_{\delta}$ has a presentation associated with the above splitting (since we know that $\delta$ is a Dehn twist), the algorithm will terminate at this presentation in finite time. 

\end{proof}

\subsubsection{The algorithm}

We will need the following result about the existence of an automorphism carrying a list of elements to another in finitely generated free products of free abelian groups. The following result holds for more general toral relatively hyperbolic groups. 

\begin{lemma} [\cite{Whitehead_for_toral_relatively_hyperbolic_groups}]
    \label{lemma: simultaneous whitehead algorithm}
    Let $G$ be a free product of finitely many finitely generated free abelian groups. There is an algorithm that decides, given two finite lists $(u_1,\ldots,u_n)$ and $(v_1,\ldots,v_n)$ of elements of $G$, whether there exists an automorphism $\phi \in \Aut(G)$ such that $\phi(u_i) = v_i$ for all $i = 1,\ldots,n$. 
\end{lemma}

Given two separating Dehn twists $\delta, \delta'$ we go through the following steps in order to determine if they are conjuguate.

\textbf{Step 1:} Compute the splittings of $\Gamma_\delta$ and $\Gamma_{\delta'}$ using Lemma \ref{lem:Tietze_algo}. 

We now have two splittings of the form $\Gamma_\delta= (A\oplus \langle t \rangle)*_{C \oplus \langle t \rangle} (B\oplus \langle tc_0^{-1} \rangle$ and $\Gamma_{\delta'}= (A'\oplus \langle t' \rangle)*_{C' \oplus \langle t' \rangle} (B'\oplus \langle t'c_0'^{-1} \rangle)$. 

\textbf{Step 2:} We decide if $A \cong A'$, $B \cong B'$ and  $C \cong C'$.

To decide if $A \cong A'$ and $B \cong B'$, we use Kurosh's subgroup theorem (Theorem \ref{thm:Kurosh}) as follows.

As $A,B,A',B' \leq A_1 * \cdots *A_k$, there are subsets $X_{A}, X_B, X_{A'}, X_{B'} \subset G$, index sets $J_A, J_{A'}, J_B, J_{B'}$, subgroups $A_j, A'_j, B_j, B'_j$ of the free factors $G_i$ of $ G$ such that
\begin{align*} 
A=F(X_A) * (*_{j \in J_A} A_j^{\dag}),\quad & B=F(X_B) * (*_{j \in J_B} B_j^{\dag}) \\ 
   A'=F(X_{A'}) * (*_{j \in J_{A'}} A_j'^{\dag}),\quad & B'=F(X_{B'}) * (*_{j \in J_{B'}} B_j'^{\dag}).
\end{align*}

Note that the index sets $J_A,J_{A'}, J_B, J_{B'}$ and the sets $X_A, X_{A'}, X_B, X_{B'}$ are all finite by Lemma \ref{lemma: A and B are finitely generated in separating case}. Also, by the discussion at the beginning of Section \ref{subsection: Step 1}, we can assume that $C$ is maximal abelian both $A$ and $B$ and that $C'$ is maximal abelian in both $A'$ and $B'$. With these decompositions, we can now check whether $A \cong A'$, $B \cong B'$ and $C \cong C'$ as follows. For $C$ and $C'$, we simply check if their ranks are the same. If the ranks are different, output NO and stop. To check if $A \cong A'$, we check if $X_A$ and $X_{A'}$ have the same cardinality and if sets of subgroups $\{A_j\}_{j \in J_A}$ and $\{A_j'\}_{j \in J_{A'}}$ are the same up to isomorphism, that is, if for each $j \in J_A$, there exists $j' \in J_{A'}$ such that $rank (A_j) = rank (A'_{j'})$ and if for each $j' \in J_{A'}$ there exists $j \in J_A$ such that $rank (A_j) = rank (A'_{j'})$. We do the same to determine if $B \cong B'$. If $X_A$ and $X_{A'}$ have different cardinalities and if the sets of subgroups $\{A_j\}_{j \in J_A}$ and $\{A_j'\}_{j \in J_{A'}}$ are different up to isomorphism, output NO and stop. Do the same for the $B$ and $B'$ subgroups. If $A \cong A'$, $B \cong B'$ and $C \cong C'$, proceed to Step 3. 

\textbf{Step 3:} Identifying $A$ with $A'$, $B$ with $B'$ and $C$ with $C'$, using Lemma \ref{lemma: simultaneous whitehead algorithm}, decide if there exists an automorphism $\psi \in \Aut(C)$ such that $\psi(c_0) = c_0'$ and automorphisms $\psi_A \in \Aut(A)$ and $\psi_B \in \Aut(B)$ such that $\psi_A(c_0) = c_0'$ and $\psi_B(c_0) = c_0'$. If not, output NO and stop, otherwise output YES and stop. 

We now justify why the algorithm gives the correct output. First, when the algorithm outputs NO in Step 2, then one of the pairs of subgroups $(A,A')$, $(B,B')$ or $(C,C')$ are not isomorphic. By Lemma \ref{lem:upgrade_to_fiber_preserving}, it follows that $\delta$ and $\delta'$ are not conjugate. When the algorithm outputs NO in Step 3, the there cannot exist isomorphisms $\phi_A, \phi_B$ as in Lemma \ref{lem:upgrade_to_fiber_preserving}, since otherwise they would restrict to an isomorphism $C \to C'$ mapping $c_0 \mapsto c_0'$. Therefore, by Lemma \ref{lem:upgrade_to_fiber_preserving} $\delta$ and $\delta'$ are not conjugate in $\Aut(G)$. When the algorithm outputs YES in Step 3, this means that $A \cong A'$, $B \cong B'$ and there is an isomorphism $\phi_C:C \to C'$ mapping $c_0$ to $c_0'$. As we assumed at the beginning of Section \ref{subsection: Step 1}, $C$ and $C'$ are maximal abelian both $A \cong A'$ and $B \cong B'$ (where we now identify $A$ with $A'$ and $B$ with $B'$ since they are isomorphic). If $C$ and $C'$ are both cyclic, then the existence of the automorphisms $\psi_A, \psi_B$ implies that $\delta$ and $\delta'$ are conjugate by Lemma \ref{lem:upgrade_to_fiber_preserving}. Otherwise, if $C$ and $C'$ are not cyclic, then they are free factors of $A,B$ (respectively, $A',B'$) and hence writing $A = C * D = C' * D'$ for free factors $D,D'$ of $A$ (which must be isomorphic since $C$ and $C'$ are isomorphic), we obtain an automorphism $\phi_A$ of $A$ taking $C$ to $C'$ by taking the free product of $\phi_C$ with any isomorphism $\phi_D:D \to D'$. Similarly, decomposing $B = C * E = C' * E'$ and defining $\phi_B = \phi_C * \psi_E$ for any isomorphism $\psi_E : E \to E'$, we get that $\phi_A$ and $\phi_B$ both map $C$ to $C'$, agree on $C$, and take $c_0$ to $c_0'$. By Lemma \ref{lem:upgrade_to_fiber_preserving}, we have that $\delta$ and $\delta'$ are conjugate in $\Aut(G)$.

\section{Non-separating Dehn twists}\label{section: non-separating Dehn twists}

In this section, we fix $G$ to be a finitely generated free product of free abelian groups.

Let $G=A*_{D^{e}=C}$ be a splitting of $G$ as an HNN extension over an abelian subgroup. A \textit{non-separating Dehn-twist} $\delta$ for $G$ corresponding to this splitting is given by

$$ \delta : \begin{cases}
    a \mapsto a & \forall a \in A \\
    e \mapsto ew_{1} 
\end{cases} $$
where $w_{1}\in D$ and $ew_{1}e^{-1}=:w_{2}\in C$.

\subsection{Step I: Finding a canonical splitting of the mapping torus}
First, as in the case of separating Dehn twists, we find an associated splitting of the mapping torus $\Gamma_{\delta} = G \rtimes_{\delta} \langle t \rangle$ of $G$ with respect to $\delta$. 
We have: 
\begin{gather*}
    \Gamma_{\delta}
    = G\rtimes_{\delta} \langle t \rangle  \\
    = \langle A,e,t \vert  D^{e}=C, \, [t,a]=1 \space \forall a\in A,\, tet^{-1}=ew_{1} \rangle \\
    = \langle A , e, t \vert  D^{e}=C,\, [t,a]=1 \space\forall a\in A, \, ete^{-1}=tw_{2}^{-1} \rangle \\
    = (A\oplus \langle t\rangle) *_{(D\oplus \langle t \rangle)^{e} = (C\oplus \langle tw_{2}^{-1} \rangle)}
\end{gather*}

Denoting $B$ the maximal abelian subgroup of $G$ containing $D$, we note that $\delta$ acts on $B$ since $\delta$ is an automorphism of $G$ that fixes $D$ pointwise. 

We denote $\tilde{A} = A \oplus \langle t \rangle, \tilde{D} = D \oplus \langle t \rangle, \tilde{C} = C \oplus \langle tw_2^{-1} \rangle = C \oplus \langle t \rangle$ and $\tilde{B} = B \rtimes \langle t \rangle$. The following lemma is a consequence of \cite[Theorem 1.3]{Finiteness_properties_graphs_of_groups}.

\begin{lemma}
    \label{lemma: A is finitely generated}
    Let $G$ be a free product of free abelian groups. Assume that $G$ is finitely generated. Suppose that $G$ splits as an HNN extension $G = A*_{D^e = C}$ with $D$ and $C$ abelian. Then $A$ is finitely generated.
\end{lemma}

When $\delta$ is a separating Dehn twist, the action of mapping torus $\Gamma_{\delta}$ on the Bass--Serre tree of the its associated splitting is 2-acylindrical. However, as we illustrate in the next section, this fails to be the case when $\delta$ is non-separating, forcing us to pass to a new construction called the \textit{tree of cylinders} of the splitting.

\subsection{Failure of acylindricity for the splitting of $\Gamma$ in the non-separating case}

Unlike in the separating Dehn twist case, the Bass-Serre tree corresponding to the above splitting of $\Gamma_{\delta}$ might not be $2$-acylindrical. Here is an example.

Let $G = \langle a,b \rangle * \langle e \rangle \cong \Z^2 * \Z$. We can express this an HNN extension over a cyclic edge group as follows: 

$$G = (\langle a,b \rangle * \langle d \rangle)*_{\langle d \rangle^e = \langle b \rangle}$$

Denote $A = \langle a,b \rangle * \langle d \rangle$, $C = \langle b \rangle$, $D = \langle d \rangle$ in our above notation. 

Consider the Dehn twist $\delta$ defined by $\delta(e) = ed$ (so that $w_1 = d$ and $w_2 = ede^{-1} = b$ in our above notation). 

\begin{claim}
    The above splitting of $G$ is 2-acylindrical.
\end{claim}

\begin{proof}
    Consider a geodesic segment of length 3 in the associated Bass--Serre tree: $(gD, D, hD)$ (it suffices to consider this particular segment because the action of $G$ on edges in the Bass--Serre tree is transitive). We have to show that $D^g \cap D \cap D^h = 1$. We either have that $g,h$ are rotations about the vertices $e^{-1}A, A$, respectively or $g,h$ are translations respectively taking the vertices $A$ to $e^{-1}A$ and $e^{-1}A$ to $A$, so that $h \in eA \setminus D$ and $g \in e^{-1}A \setminus D$. 

    First, suppose that $g,h$ are both elliptic. Then $g \in A^{e^{-1}} = \langle a^{e^{-1}}, d \rangle * \langle d^{e^{-1}} \rangle$ and $h \in A$. We then have $D \cap D^g$ is either equal to $D$ or is trivial, being non-trivial if and only if $g \in \langle a^{e^{-1}}, d \rangle$. Similarly, $D \cap D^h$ is either $D$ or is trivial, being non-trivial if and only if $h \in D$ since $D$ is malnormal in $A$. Therefore, $D^g \cap D \cap D^h \neq 1$ if and only if $g \in \langle a^{e^{-1}}, d \rangle$ and $h \in D$. But then since $h \in D$, we have $hD = D$, so the segment $(gD, D, hD)$ has length at most 2, a contradiction. Therefore, we can assume that at least one of $g,h$ is hyperbolic. 

    However, if $k \in G$ is hyperbolic, then if $D^k \cap D \neq 1$, letting $B$ denote the maximal abelian subgroup of $G$ containing $D$, we would obtain $B^k \cap B \neq 1$, which implies that $k \in B$ since $B$ is malnormal (as $G$ is CSA). In this case, we have that $B = \langle a,b \rangle^{e^{-1}} < A^{e^{-1}}$, so $k \in A^{e^{-1}}$ and hence is elliptic, a contradiction. Therefore, if $g$ or $h$ is hyperbolic, we must have $D^g \cap D = 1$ or $D^h \cap D = 1$. 

    Thus, in all cases, we obtain that $D^g \cap D \cap D^h = 1$, so the splitting is 2-acylindrical.
    
\end{proof}

\begin{claim}
    \label{claim: non-acylindrical mapping torus}
    The associated splitting of the mapping torus $\Gamma = G \rtimes_{\delta} \langle t \rangle$ given by: 

    $$\Gamma = (A \oplus \langle t \rangle)*_{(D \oplus \langle t \rangle)^e = C \oplus \langle tw_2^{-1} \rangle?}$$

    is not 2-acylindrical. In fact it is not $k$-acylindrical for any $k \geq 0$.
\end{claim}

\begin{proof}
    Denote $\tilde{A} = A \oplus \langle t \rangle$ and $\tilde{D} = D \oplus \langle t \rangle$. Consider $g = e^{-1}ae$ and $h = a$. Then $g \in \mathrm{Stab}_{\Gamma}(e^{-1}\tilde{A}) \setminus \mathrm{Stab}_{\Gamma}(\tilde{D})$ and $h \in \mathrm{Stab}_{\Gamma}(\tilde{A}) \setminus \mathrm{Stab}_{\Gamma}(\tilde{D})$, so that $(g\tilde{D}, \tilde{D}, h\tilde{D})$ is a geodesic edge path of length 3 in the associated Bass--Serre tree. We have: 

    \begin{align*}
        \tilde{D} \cap \tilde{D}^g &= (D \oplus \langle t \rangle) \cap (D^g \oplus \langle t^g \rangle) \\
        &= (D \oplus \langle t \rangle) \cap (D^g \oplus \langle t^{e^{-1}ae} \rangle) \\
        &= (D \oplus \langle t \rangle) \cap (D^g \oplus \langle (tb^{-1})^{e^{-1}a} \rangle) \text{, since $t^e = tw_2^{-1} = tb^{-1}$}\\
        &= (D \oplus \langle t \rangle) \cap (D^g \oplus \langle t^{e^{-1}a}(b^{-1})^{e^{-1}a} \rangle) \\
        &=(D \oplus \langle t \rangle) \cap (D^g \oplus \langle (td)d^{-1}\rangle) \text{, since $t^{e^{-1}} = tw_1 = td$ and $b^{e^{-1}} = d$} \\
        &= (D \oplus \langle t \rangle) \cap (D^g \oplus \langle t\rangle) \\
        &\supseteq \langle t \rangle 
    \end{align*}

    Similarly, we have $\langle t \rangle \cap \langle t \rangle^h = \langle t \rangle \cap \langle t \rangle^a = \langle t \rangle$ since $t$ centralizes $A$. Thus, $\tilde{D} \cap \tilde{D}^h \supseteq \langle t \rangle$. We conclude that $\tilde{D}^g \cap \tilde{D} \cap \tilde{D}^h \supseteq \langle t \rangle \neq 1$, so the above splitting of $\Gamma$ is not 2-acylindrical. 

    In fact, the splitting is not $k$-acylindrical for any $k$. Indeed, consider the sequence of group elements $g_1 = a^{e^{-1}}, g_2 = a^{e^{-1}}a, g_3 = a^{e^{-1}}aa^{e^{-1}}, \ldots$ and $h_1 = a, h_2 = aa^{e^{-1}}, h_3=aa^{e^{-1}}a, \ldots$ (the $g_i,h_i$ are an alternating product of $a$ and $a^{e^{-1}}$ where we alternate multiplying on the right by $a$ and $a^{e^{-1}}$). Then the sequence of edges $\gamma = (\ldots,g_2 \tilde{D}, g_1 \tilde{D}, \tilde{D}, h_1 \tilde{D}, h_2 \tilde{D}, \ldots)$ is a geodesic. Indeed, it's enough to prove that $(\tilde{D}, h_1 \tilde{D}, h_2 \tilde{D}, \ldots)$ is a geodesic, since $(\ldots,g_2 \tilde{D}, g_1 \tilde{D}, \tilde{D}) = \tilde{D} \cup a^{e^{-1}}(\tilde{D}, h_1 \tilde{D}, h_2 \tilde{D}, \ldots)$, and $\tilde{D}$ and $a^{e^{-1}}(\tilde{D}, h_1 \tilde{D}, h_2 \tilde{D}, \ldots)$ meet at a vertex, so their concatenation is a geodesic. We proceed by induction on $n$. 

    For $n=1$, since $a \notin \tilde{D}$, we have that $(\tilde{D}, h_1\tilde{D}) = (\tilde{D}, a\tilde{D})$ is a geodesic. 

    Suppose that $(\tilde{D}, h_1 \tilde{D}, h_2 \tilde{D}, \ldots, h_n \tilde{D})$ is a geodesic. Then $(h_n \tilde{D}, h_{n+1} \tilde{D}) = h_n(\tilde{D}, a\tilde{D})$ or $(h_n \tilde{D}, h_{n+1} \tilde{D}) = h_n(\tilde{D}, a^{e^{-1}}\tilde{D})$, which are geodesics in each case. Since $(\tilde{D}, h_1 \tilde{D}, h_2 \tilde{D}, \ldots, h_{n+1} \tilde{D}) = (\tilde{D}, h_1 \tilde{D}, h_2 \tilde{D}, \ldots, h_n \tilde{D}) \cup (h_n \tilde{D}, h_{n+1}\tilde{D})$ and since $(\tilde{D}, h_1 \tilde{D}, h_2 \tilde{D}, \ldots, h_n \tilde{D}), (h_n \tilde{D}, h_{n+1}\tilde{D})$ meet at a single vertex (since we do not have $h_{n+1}\tilde{D} = h_{n-1} \tilde{D}$, as otherwise $a a^{e^{-1}}\tilde{D} = \tilde{D}$ or $a^{e^{-1}} a \tilde{D} = \tilde{D}$, a contradiction since neither $aa^{e^{-1}}$ nor $a^{e^{-1}}a$ are in $\tilde{D}$), we have that $(\tilde{D}, h_1 \tilde{D}, h_2 \tilde{D}, \ldots, h_{n+1} \tilde{D})$ is a geodesic. 

    However, since $a$ and $a^{e^{-1}}$ both commute with $t$, we have that $t \in P\mathrm{Stab}_{\Gamma}(\gamma)$. Hence, the splitting is not $k$-acylindrical for any $k$. 
\end{proof}

That is why the non-separating Dehn twist case requires a modified approach to find a canonical tree for solving the conjugacy problem. We do that by enlarging the vertex group in the HNN splitting of $G$ first and then use the \textit{tree of cylinders} construction as explained below.

\subsection{Tree of cylinders of mapping torus}

\begin{thm}
    \label{thm: Tree of cylinders of mapping torus}
    Let $G$ be a finitely generated torsion-free CSA group which splits over a free abelian subgroup as $G = A*_{D^e=C}$. Denote $B$ the maximal abelian subgroup of $G$ containing $D$. Let $\delta$ be a non-separating Dehn twist of this HNN extension and let $\Gamma = G \rtimes \langle t \rangle$ be the associated mapping torus. Let $T$ be the Bass--Serre tree associated with the $\Gamma$-splitting $\Gamma = \tilde{A}*_{\tilde{D}^e=\tilde{C}}$. The relation on edges of $T$ defined by: 
    
    $$E \sim F \iff \langle \Gamma_E, \Gamma_F \rangle \text{ is abelian}$$

    is an equivalence relation. 
\end{thm}

To prove Theorem \ref{thm: Tree of cylinders of mapping torus}, we will need the following lemma. We use the notation from the beginning of Section \ref{section: non-separating Dehn twists}.

\begin{lemma}
    \label{lemma: action of delta on B}
    For each $b \in B$, there exists $m \in \Z$ such that $\delta(b) = bw_1^m$. 
\end{lemma}

\begin{proof}
    Let $b \in B$. We first show that there exist $b_1,\ldots,b_n \in B$ such that either $b_i \in A \cup Ae$ for all $i$ or $b_i \in A^{e^{-1}} \cup e^{-1}A$ for all $i$ and $b = b_1 b_2 \cdots b_n$.
    
    In the tree $T_G$, connect the edges $D$ and $bD$ by a geodesic path $\gamma$ (where $D$ here refers to the edge in the fundamental domain of $T_G$ having vertices $e^{-1}A$ and $A$). Write $\gamma = (g_0D = D, g_1D, \ldots, g_nD = bD)$ ($g_0 = 1$ and $g_n = b$). Since the extremal edges of $\gamma$ are $D$ and $bD$, we have that $P\mathrm{Stab}(\gamma) = D \cap D^b = D$ since $b \in B$. Thus, for all $i$, we have $D \cap D^{g_i} = P\mathrm{Stab}([D, g_i D]) \geq P\mathrm{Stab}(\gamma) = D$, so that $g_i \in B$ for all $i$. Set $b_i = g_{i-1}^{-1}g_{i}$ for all $i = 1,\ldots,n$. Note that $b_i \in B$ for all $i$ since $g_i \in B$ for all $i$. Since $g_{i-1}D$ and $g_iD$ are adjacent edges in $T_G$, we have that $g_{i-1}^{-1}g_iD$ is adjacent to $D$. Therefore, either $b_i = g_{i-1}^{-1}g_i \in A \cup Ae$ or $b_i = g_{i-1}^{-1}g_i \in A^{e^{-1}} \cup e^{-1}A$. Since the edges $g_iD$ are at increasing distance from $D$ and since $\gamma$ is a geodesic edge path, we must have either $b_i = g_{i-1}^{-1}g_i \in A \cup Ae$ or $b_i = g_{i-1}^{-1}g_i \in A^{e^{-1}} \cup e^{-1}A$ for all $i$. 

    Given such $b_i$, we have $\delta(b_i) = b_i d_i$ for $d_i \in \{1, w_1, w_1^{-1}\}$, since if $b_i \in A \cup Ae$, then $\delta(b_i) = b_i$ (if $b_i \in A$) or $\delta(b_i) = b_i w_1$ (if $b_i \in Ae$), and if $b_i \in A^{e^{-1}} \cup e^{-1}A$, then $\delta(b_i) = b_i$ (if $b_i \in A^{e^{-1}}$) or $\delta(b_i) = b_i w_1^{-1}$ (if $b_i \in e^{-1}A$). Therefore, if $b \in B$, then we have that $\delta(b) = b w_1^m$ for some $m \in \Z$.
\end{proof}

\begin{proof} [Proof of Theorem \ref{thm: Tree of cylinders of mapping torus}]
    Since edge stabilizers of $T$ are abelian, it is clear that $\sim$ is reflexive and symmetric, so it remains to prove that it is transitive. Recall that edge stabilizers of $T$ are of the form $\tilde{D}^g = D^g \oplus \langle t^g \rangle$ for $g \in \Gamma$. Note that $\sim$ descends to a conjugation invariant relation on edge stabilizers of $T$, so it suffices to prove that if $\langle \tilde{D}^g, \tilde{D}\rangle$ and $\langle \tilde{D}^h, \tilde{D}\rangle$ are abelian then so is $\langle \tilde{D}^{g^{-1}h}, \tilde{D}\rangle$. 

    We begin by finding a necessary and sufficient condition on $g$ to ensure that $\langle \tilde{D}^g, \tilde{D}\rangle$  is abelian. 

    For $\langle \tilde{D}^g, \tilde{D}\rangle$ to be abelian, we need that each element of $D^g$ as well as $t^g$ commute with each element of $\tilde{D}$. Writing $g = g_0 t^n$ for some $n \in \Z$ and $g_0 \in G$, we see that $\tilde{D}^g = \tilde{D}^{g_0}$, so we can assume that $g \in G$. 

    For $D^g$ to centralize $\tilde{D}$, we need $D^g$ to centralize $D$ and $t$. We have that $D^g$ centralizes $D$ if and only if $g \in B$, since if $D^g$ centralizes $D$, then for any $d \in D \setminus 1$ we have $d^g \in B$ and so $1 \neq d \in B \cap B^{g^{-1}}$, giving that $g \in B$ by malnormality of $B$ in $G$ (since $G$ is CSA). Conversely, if $g \in B$, then $D^g = D$ which centralizes $D$ and $t$. Therefore, $D^g$ centralizes $\tilde{D}$ if and only if $g \in B$. We now check that for each $g \in B$, $t^g$ also centralizes $\tilde{D}$. We have $t^g = gtg^{-1} = g(g^{-1})^t t$. If $g \in B$, by Lemma \ref{lemma: action of delta on B}, we have that $g^t = \delta(g) = g w_1^n$ for some $n \in \Z$. Thus, $g (g^{-1})^t = g (gw_1^n)^{-1} = w_1^{-n}$, so that $t^g = w_1^{-n} t$, which centralizes both $D$ and $t$. 

    Therefore, for $g \in G$, we have that $\tilde{D}^g$ centralizes $\tilde{D}$ if and only if $g \in B$. Thus, for $g \in \Gamma$, $\tilde{D}^g$ centralizes $\tilde{D}$ if and only if $g \in \tilde{B} = B \rtimes \langle t \rangle$ (this shows, in fact, that $\langle \tilde{D}^g, \tilde{D} \rangle$ being abelian implies that $\tilde{D}^g=\tilde{D}$). 

    If $\tilde{D}^g$ and $\tilde{D}^h$ both centralize $\tilde{D}$, then $g,h \in \tilde{B}$ and hence $g^{-1}h \in \tilde{B}$, therefore, $\tilde{D}^{g^{-1}h}$ centralizes $\tilde{D}$ and so $\langle \tilde{D}^{g^{-1}h}, \tilde{D} \rangle$ is abelian. We conclude that $\sim$ is an equivalence relation.
\end{proof}

We define a \textit{cylinder} of $T$ to be an equivalence class of an edge. Since no two distinct cylinders of $T$ can contain an edge, two distinct cylinders of $T$ meet in at most one vertex. By \cite[Section 4.5]{Limit_groups_groups_acting_freely}, it follows that the \textit{dual graph} $T_C$ of $T$, defined as the bipartite graph having white vertices the vertices of $T$ contained in at least two cylinders and the black vertices the cylinders of $T$, with a white vertex $v$ adjacent to a black vertex $b$ if and only if $v \in b$ in $T$, is a tree. We refer to $T_C$ as the \textbf{tree of cylinders of $T$}. 

Since $\sim$ is invariant under translation by $\Gamma$, we have that $T_C$ is a $\Gamma$-tree. Since $T$ has only one orbit of vertices and edges, $T_C$ has one orbit of white vertices and one orbit of black vertices. 

Note that the stabilizer of the cylinder containing $\tilde{D}$ is $\tilde{B} := B \rtimes \langle t \rangle$, since as we showed above, for $g \in \Gamma$ we have $\langle \tilde{D}^g, \tilde{D} \rangle$ abelian if and only if $g \in \tilde{B}$. Since there is only one orbit of edges in $T$, there is only one orbit of cylinders in $T_C$ and the stabilizer of the cylinder containing $\tilde{D}$ is $\tilde{B}$. Thus, by the orbit stabilizer theorem, we can identify cylinders with left cosets of $\tilde{B}$ in $\Gamma$. 

We have that $T_C$ has at most two orbits of edges, since each edge $E$ of $T_C$ has a translate $gE$ that is adjacent to $\tilde{B}$. Then the white vertex $v$ of $gE$ is in the cylinder containing $\tilde{D}$, so $v = he^{-1} \tilde{A}$ or $v = h \tilde{A}$ for some $h \in \tilde{B}$, so that $h^{-1}gE = [e^{-1}\tilde{A}, \tilde{B}]$ or $[\tilde{B}, \tilde{A}]$. 

\subsection{Tree of cylinders is acylindrical for splittings with maximal vertex group}
\label{section: tree of cylinders for non-separating Dehn twist}

\begin{lemma}
    \label{lem: Fixed point set of powers}
    Let $G$ be a finitely generated torsion-free CSA group which splits over a free abelian subgroup as $G = A*_{D^e=C}$. Let $\delta$ be a non-separating Dehn twist of this HNN extension, with $\delta(e) = ew$ for some $w \in D \setminus 1$. Let $B$ be the maximal abelian subgroup of $G$ containing $D$. Then for each $n \neq 0$, we have $\mathrm{Fix}(\delta \vert_B^n) = \mathrm{Fix}(\delta \vert_B)$.
\end{lemma}

\begin{proof}
    By Lemma \ref{lemma: action of delta on B}, for each $b \in B$ we have that $\delta(b) = bw^m$ for some $m \in \Z$. Thus, for $n \neq 0$, we have $\delta^n(b) = bw^{mn}$. If $\delta^n(b) = b$, then $bw^{mn} = b$, which implies $w^{mn} = 1$, so that $m=n=0$ since $w$ has infinite order. Thus, $\delta(b) = b$.
\end{proof}

\begin{thm}
    \label{thm: Enlarging the splitting}
    Let $G$ be a finitely generated torsion-free CSA group and let $\delta$ be a non-trivial, non-separating Dehn twist of $G$. Then $G$ splits as an HNN extension $G = A*_{D^e=C}$ with $A = \mathrm{Fix}(\delta)$, $D$ free abelian and maximal abelian in $A$ and such that $\delta(e) = ew_1$ for some $w_1 \in D \setminus 1$. 
\end{thm}

\begin{proof}
    Let $G = {A_0}*_{D_0^e=C_0}$ be some HNN extension on which $\delta$ acts as a non-separating Dehn twist (i.e.\ such that $\delta\vert_{A_0} = \mathrm{id}$, $D_0$ is free abelian and $\delta(e) = ew_1$ for some $w_1 \in D$). Denote $A = \mathrm{Fix}(\delta)$ and let $D$ (respectively, $C$) be the maximal abelian subgroup of $D_0$ (respectively, $C_0$) in $A$. Denote $w_2 = w_1^e$. We will prove that $G = A*_{D^e=C}$.

    First, note that $D^e < A$ since if $d \in D$, then $\delta(ede^{-1}) = ew_1 d w_1^{-1}e^{-1} = ede^{-1}$ since $d$ commutes with $w_1$. Therefore, since $C_0 = D_0^e$ in $A_0$, we have that $C = D^e$ also holds in $A$. 

    Consider the graph $\mathcal{G}$ with vertices left cosets $gA$ of $A$ in $G$ and an edge joining the cosets $ge^{-1}A$ and $gA$ for each $g \in G$. We will prove that $\mathcal{G}$ is a tree. This will then imply that $G = A*_{D^e=C}$ since $G$ acts by isometries on $\mathcal{G}$ with one orbit of vertices and one orbit of edges, and the stabilizer of the vertex $A$ is $A$ while the stabilizer of the fundamental edge $[e^{-1}A, A]$ is $A^{e^{-1}} \cap A$. Note that $A^{e^{-1}} \cap A = D$ since $D < A^{e^{-1}} \cap A$ by above, and conversely if $a \in A$ with $a^e \in A$ then $ eae^{-1}=\delta(eae^{-1}) = ew_1 a w_1^{-1}e^{-1} $, so $[a,w_1]=1$ and hence $w_1 \in D^a \cap D$, so that $a \in D$ by malnormality of $D$ in $A$ (since $D$ is maximal abelian in $A$ and $A$ is CSA). 

    Let $T$ be the Bass--Serre tree of the splitting $G = {A_0}*_{D_0^e=C_0}$. Then in $T$ we have induced subgraphs corresponding to the cosets of $A$ in $G$, since $A_0 < A$. 
    
    Note that the subgraphs of $T$ corresponding to the cosets $hA$ are convex, hence subtrees of $T$. Indeed, note that $\delta$ is an isometry of $T$ since given an edge $[gA_0, ghe^{\pm 1}A_0]$ for $h \in A_0$ and $g \in G$, applying $\delta$ to the endpoints of this edge yields the vertices $\delta(gA_0) = \delta(g)A_0$ and 

    $$\delta(gheA_0) = \delta(g) \delta(h) ew_1 A_0=\delta(g) \delta(h) e A_0$$ or $$\delta(ghe^{-1}A_0) = \delta(g) \delta(h) w_1^{-1}e^{-1} A_0=\delta(g) \delta(h) e^{-1}w_2^{-1} A_0=\delta(g) \delta(h) e^{-1} A_0$$, and $\delta(h) \in A_0$ so that $[\delta(gA_0), \delta(ghe^{\pm1}A_0)]$ is an edge. Then if $gA_0 \subseteq A$, we have $\delta([A_0, gA_0]) = [A_0, gA_0]$, so that $[A_0, gA_0] \subset A$, so that $A$ (and hence also coset subgraphs $hA$ in $T$) is convex.

    We have that the cosets $A$ and $hA$ are adjacent in $\mathcal{G}$ if and only if the subgraphs $A$ and $hA$ are at distance 1 in $T$ (i.e.\ are joined by an edge in $T$). Indeed, suppose that $A$ and $hA$ are adjacent in $\mathcal{G}$, so that $hA = g e^{\pm1}A$ for some $g \in A$. Then there exists $a \in A$ and $g \in A$ such that $h = g e^{\pm 1} a$. We can decompose $a = a_0 g_0 $ for some $a_0 \in A_0$ and $g_0 \in A$. Then we obtain $h = g e^{\pm 1} a_0 g_0$, so that $h g_0^{-1} = g e^{\pm 1}a_0$, so that $hg_0^{-1}A_0 = g e^{\pm 1}A_0$, so that the cosets $gA_0 \subset A$ and $hg_0 A_0 \subset hA$ are adjacent in $T$. Conversely, if $A$ and $hA$ are adjacent in $T$, then there exists $g_1,g_2 \in A$ such that $g_1 A_0$ and $hg_2 A_0$ are adjacent in $T$, so that there exists $g \in G$ such that $hg_2A_0 = g_1 ge^{\pm 1}A_0$, so that $hA = g_1 g e^{\pm 1}A$, so that $hA$ is adjacent to $g_1A=A$ in $\mathcal{G}$. 

    We then have that $\mathcal{G}$ is connected since given a coset $hA$, we can join any pair of minimal distance vertices of the subgraphs $A$ and $hA$ by a geodesic $\gamma$ in $T$ and then cover the vertices of $\gamma$ by pairwise distinct cosets $h_1 A, \ldots, h_nA$ with $h_1 A=A$ and $h_nA=hA$. Since the cosets $h_i A$ cover the vertices of $\gamma$, we have that minimal distance cosets $h_iA$ and $h_jA$ are joined by an edge, hence are adjacent in $\mathcal{G}$. Thus, $h_1A, \ldots, h_n A$ is an edge path in $\mathcal{G}$ from $A$ to $hA$. 

    We have that $\mathcal{G}$ is simply connected since if $A=h_1 A, \ldots, h_nA=A$ is a simple closed path in $\mathcal{G}$, then connecting the minimal distance vertices $v_i, v_{i+1}$ of cosets $h_iA, h_{i+1}A$ by edges in $T$ and connecting $v_{i-1}, v_i$ by the geodesic $[v_{i-1},v_i] \subset h_iA$, since $h_i A$ is convex, we obtain a simple closed path in $T$, contradicting that $T$ is a tree. 

    Therefore, $\mathcal{G}$ is a tree, and we conclude that $G = A*_{D^e=C}$.
\end{proof}

We will call a splitting as in Theorem \ref{thm: Enlarging the splitting} a \textbf{maximal splitting}.

From Theorem \ref{thm: Enlarging the splitting}, if $G$ is a free product of finitely many finitely generated free abelian groups, if $D$ and $C$ are not cyclic, then we can assume that either $D = C$ or that $D*C$ is a free factor of $A$. 

Indeed, by Theorem \ref{thm:Kurosh}, we have that $A$ is a free product of free abelian groups, and hence since $D$ and $C$ are maximal abelian in $A$ by the proof of Theorem \ref{thm: Enlarging the splitting}, we have that $D$ and $C$ are free factors of $A$, or are cyclic, generated by a primitive element. Recall also that $A$ is finitely generated by Lemma \ref{lemma: A is finitely generated}.

By the Grushko decomposition theorem, we can decompose $A$ as a free product in two ways: 

$$A = D * R_1 * \cdots * R_n = C * S_1 * \cdots * S_n$$

for free abelian groups $R_i, S_j$, and $C$ is conjugate in $A$ to $D$ or to $R_i$ for some $i$. 

If $C = D^a$ for some $a \in A$, then putting $\tilde{e} = a^{-1}e$, we have that $D^{\tilde{e}} = C^{a^{-1}} = D$. Hence, $\tilde{e}$ normalizes $D$ and hence centralizes $D$, since $D$ is abelian. We then have the splitting: 

$$G = A*_{D^{\tilde{e}} = D}$$

and $\delta$ similarly acts as a non-separating Dehn twist on this splitting since $\delta(\tilde{e}) = \delta(a^{-1}e) = a^{-1}ew = \tilde{e}w$. Thus, we can work with this splitting instead and replace $C$ with $D$. 

If $C = R_i^a$ for some $i$ and some $a \in A$, then we similarly define $\tilde{e} = a^{-1}e$, giving that $D^{\tilde{e}} = C^{a^{-1}} = R_i$, and we obtain the splitting: 

$$G = A*_{D^{\tilde{e}} = R_i}$$

on which $\delta$ acts as a non-separating Dehn twist. Working with this splitting instead, we can replace $C$ with $R_i$ and hence assume that $D * C$ is a free factor of $A$. 

\begin{thm}
    \label{thm:tree of cylinders max splitting is acylindrical}

    Let $G$ be a finitely generated torsion-free CSA group which splits over a free abelian subgroup as $G = A*_{D^e=C}$. Let $\delta$ be a non-separating Dehn twist of this HNN extension, with $\delta(e)=ew_1$ for some $w_1 \in D$. Suppose further that $A = \mathrm{Fix}(\delta)$. Then the associated tree of cylinders splitting of $\Gamma = G \rtimes_{\delta} \langle t \rangle$ is 2-acylindrical. 
    
\end{thm}

\begin{proof}
    Recall that the tree of cylinders $T_C$ of the associated splitting  $\Gamma = \tilde{A}_{\tilde{D}^e = \tilde{C}}$ has two orbits of vertices (the orbit of the white vertex $\tilde{A}$ and the orbit of the black vertex corresponding to the cylinder containing $\tilde{D}$, which we identify with its stabilizer $\tilde{B} = B \rtimes \langle t \rangle$, where $B$ is the maximal abelian subgroup of $G$ containing $D$), and at most two orbits of edges (the orbits of the edges $E=[e^{-1}\tilde{A}, \tilde{B}]$ and $F=[\tilde{B}, \tilde{A}]$).

    Note that $B \cap A^{e^{-1}} = D = B \cap A$. Indeed, since $\delta$ fixes $B \cap A^{e^{-1}}$ pointwise, as if $e^{-1}ae \in B$ for $a \in A$ then $\delta(e^{-1}ae) = w_1^{-1}e^{-1}aew_1 = e^{-1}ae$ since $e^{-1}ae \in B$. Therefore, since $A = \mathrm{Fix}(\delta)$, we have $B \cap A^{e^{-1}} < A$. Thus, $B \cap A^{e^{-1}} < A \cap A^{e^{-1}}=D$. Conversely, $D < B \cap A^{e^{-1}}$, so we obtain $B \cap A^{e^{-1}} = D$. Similarly, if $a \in B \cap A$, then $\delta(eae^{-1}) = ew_1 a w_1^{-1}e^{-1} = eae^{-1}$, so $(B \cap A)^{e} < A$, which implies $B \cap A < A^{e^{-1}}$, which yields $B \cap A < A \cap A^{e^{-1}}=D$, so that $B \cap A = D$ (as conversely $D < B$ and $D < A$). Therefore, the stabilizers of the edges $E$ and $F$ are respectively, $\tilde{B} \cap \tilde{A}^{e^{-1}} = \tilde{D}$ and $\tilde{B} \cap \tilde{A} = \tilde{D}$. 

    Now if $\gamma = (E_1,E_2,E_3)$ is an edge path of length 3 in $T_C$, then the middle edge of $\gamma$ is either in the orbit of $E$ or in the orbit of $F$. Therefore, we can assume that the middle edge $E_2$ of $\gamma$ is $E$ or $F$. 

    First, suppose the middle edge of $\gamma$ is $E$. We consider the following cases: 

    \begin{enumerate}
        \item $E_1 = gE$ and $E_3 = hE$ for some $g,h \in \Gamma$. Then $g \in \tilde{A}^{e^{-1}} \setminus \tilde{D} = \tilde{A}^{e^{-1}} \setminus \tilde{B}$, so that $\tilde{D}^g \cap \tilde{D} = \langle t^g \rangle \cap \langle t \rangle$. Also, $h \in \tilde{B} \setminus \tilde{D}$, so that $\tilde{D}^h \cap \tilde{D} = D \oplus (\langle t^h \rangle \cap \langle t \rangle)$. In order for $\langle t^h \rangle \cap \langle t \rangle \neq 1$, we need $[h,t^n] = 1$ for some $n \neq 0$, i.e.\ $h \in \mathrm{Fix}(\delta^n) \oplus \langle t \rangle = \mathrm{Fix}(\delta) \oplus \langle t \rangle$ by Lemma \ref{lem: Fixed point set of powers}. Thus, we need $h \in \tilde{A}$ for $\langle t^h \rangle \cap \langle t \rangle \neq 1$. But $h \notin \tilde{A}$ since otherwise $h \in \tilde{B} \cap \tilde{A} = \tilde{D}$. Thus, $\langle t \rangle^h \cap \langle t \rangle = 1$. Thus, we obtain $\tilde{D}^g \cap \tilde{D} \leq \langle t \rangle$ while $\tilde{D}^h \cap \tilde{D} \leq D$, so that $\mathrm{Stab}(\gamma) = \tilde{D}^g \cap \tilde{D} \cap \tilde{D}^h = 1$. 

        \item $E_1 = gE$ and $E_3 = hF$. As above, we obtain $\mathrm{Stab}(E_1) \cap \mathrm{Stab}(E_2) \leq \langle t \rangle$. On the other hand, $h \in \tilde{B} \setminus \tilde{A}$, so that $\langle t \rangle^h \cap \langle t \rangle = 1$, so that $\mathrm{Stab}(E_2) \cap \mathrm{Stab}(E_3) \leq D$. Thus, $\mathrm{Stab}(\gamma) = 1$.

        \item $E_1 = gF$ and $E_3 = hE$. As above, we obtain $\mathrm{Stab}(E_3) \cap \mathrm{Stab}(E_2) \leq D$. We have $g \notin \tilde{B}$ since $E_1$ does not contain the vertex $\tilde{B}$. Hence, $\mathrm{Stab}(E_1) \cap \mathrm{Stab}(E_2) = \tilde{D}^g \cap \tilde{D} \leq \langle t \rangle$. Thus, $\mathrm{Stab}(\gamma) = 1$

        \item $E_1 = gF$ and $E_3 = hF$. Then $\mathrm{Stab}(E_1) \cap \mathrm{Stab}(E_2) \leq \langle t \rangle$ while $\mathrm{Stab}(E_3) \cap \mathrm{Stab}(E_2) \leq D$. Thus, $\mathrm{Stab}(\gamma) = 1$.
        
    \end{enumerate}

    Now suppose the middle edge of $\gamma$ is $F$. We consider the following cases: 

    \begin{enumerate}
        \item $E_1 = gF$ and $E_3 = hF$ for some $g,h \in \Gamma$. Then as above $\mathrm{Stab}(E_1) \cap \mathrm{Stab}(E_2) \leq D$ while $\mathrm{Stab}(E_3) \cap \mathrm{Stab}(E_2) \leq \langle t \rangle$, so that $\mathrm{Stab}(\gamma) = 1$. 

        \item $E_1 = gF$ and $E_3 = hE$. Then $\mathrm{Stab}(E_1) \cap \mathrm{Stab}(E_2) \leq D$ while $\mathrm{Stab}(E_3) \cap \mathrm{Stab}(E_2) \leq \langle t \rangle$ since $h \notin \tilde{B}$. Thus, $\mathrm{Stab}(\gamma) = 1$.

        \item $E_1 = gE$ and $E_3 = hF$. Then $\mathrm{Stab}(E_1) \cap \mathrm{Stab}(E_2) \leq D$ as $g \in \tilde{B} \setminus \tilde{A}$ while $\mathrm{Stab}(E_3) \cap \mathrm{Stab}(E_2) \leq \langle t \rangle$ as $h \in \tilde{A} \setminus \tilde{B}$. Thus, $\mathrm{Stab}(\gamma) = 1$

        \item $E_1 = gE$ and $E_3 = hE$. Then again $\mathrm{Stab}(E_1) \cap \mathrm{Stab}(E_2) \leq D$ since $g \in \tilde{B} \setminus \tilde{A}$ and  $\mathrm{Stab}(E_3) \cap \mathrm{Stab}(E_2) \leq \langle t \rangle$ since $h \notin \tilde{B}$. Thus, $\mathrm{Stab}(\gamma) = 1$.
        
    \end{enumerate}

    Therefore, in every case, we obtain $\mathrm{Stab}(\gamma) = 1$, so $T_C$ is 2-acylindrical.

\end{proof}

\subsection{Canonicity of the tree of cylinders and algorithm to solve conjugacy}

\begin{lemma}
    \label{lem: canonicity lemma for tree of cylinders}
    Let $G$ be a finitely generated CSA group that splits over a free abelian subgroup as $G = A*_{D^e=C}$. Let $\delta$ be a non-separating Dehn twist of this HNN extension, with $\delta(e)=ew_1$ for some $w_1 \in D$. Suppose further that the splitting is maximal and that $A$ is non-abelian. Let $T_C$ be the tree of cylinders of the associated splitting of the mapping torus $\Gamma = \tilde{A}*_{\tilde{D}^e = \tilde{C}}$. 

    Suppose that $S$ is a bipartite $\Gamma$-tree satisfying the following properties: 

    \begin{enumerate}
        \item Edge stabilizers of $S$ are abelian.
        \item White vertex stabilizers of $S$ are of the form $H \oplus \Z$ where $H < G$ has trivial centre.
        \item Black vertex stabilizers of $S$ are of the form $K \rtimes \Z$ where $K < G$ is abelian.
        \item $S$ is 2-acylindrical.
    \end{enumerate}

    Then $S$ is $\Gamma$-equivariantly isomorphic to $T_C$. 
\end{lemma}

\begin{proof}
    We will construct a $\Gamma$-equivariant bijection from the white vertices of $T_C$ to the white vertices of $S$ and then extend this to a $\Gamma$-equivariant tree isomorphism $T_C \to S$. 

    We first show that $\tilde{A}$ is elliptic on $S$. 
    
    \begin{lemma}
        \label{lem: tilde A fixes a unique white vertex}
        Let $T_C, S$ and $\tilde{A}$ be as in the statement of Lemma \ref{lem: canonicity lemma for tree of cylinders}. Then $\tilde{A}$ fixes a unique white vertex $u'$ of $S$. 
    \end{lemma}

    \begin{proof}[Proof of Lemma \ref{lem: tilde A fixes a unique white vertex}]
        First, we show that $\tilde{A}$ fixes a unique white vertex of $S$. Suppose for contradiction that $\tilde{A}$ was hyperbolic, with some $g \in \tilde{A}$ hyperbolic on $S$. Note that the centralizer $C_{\Gamma}(g)$ of $g$ in $\Gamma$ contains $\Z^2$. Indeed, if $g \notin \langle t \rangle$, then $\langle g, t \rangle < C_{\Gamma}(g)$ and $\langle g, t \rangle \cong \Z^2$, and if $g \in \langle t \rangle$, then taking any non-trivial $a \in A$ yields $\langle a, g \rangle < C_{\Gamma}(g)$ and $\langle a, g \rangle \cong \Z^2$. This contradicts Proposition \ref{prop: properties of groups acting acylindrically on trees} (1), since $\Gamma \curvearrowright S$ is non-elementary acylindrical. 

    Thus, $\tilde{A}$ fixes a vertex $v \in V(S)$. We have that $v$ is a white vertex of $S$. Indeed, suppose otherwise so that we have $\tilde{A} < K \rtimes \langle g_0t^m \rangle \cong K \rtimes \Z$ for some maximal abelian $K < G$ and some $g_0 \in G$. We must have $g_0 = 1$ and $m=1$, since we have $t \in K \rtimes \langle g_0t^m \rangle$, so $t = k (g_0t^m)^n$ for some $k \in K$ and $n \in \Z$, which implies $mn=1$ so $m = \pm 1$ and $g_0 = k^{-1}$. Thus, $K \rtimes \Z = K \rtimes \langle k^{-1}t^{\pm 1} \rangle = K \rtimes \langle t \rangle$. Then $A=\mathrm{proj}_G(\tilde{A}) < \mathrm{proj}_G(K \rtimes \langle t \rangle ) = K$, which implies that $A$ is abelian, a contradiction. Therefore, $\tilde{A}$ must fix a white vertex $v$ of $S$. We have that $v$ is the unique white vertex fixed by $\tilde{A}$ since if $\tilde{A}$ fixed another white vertex $v'$, then it would fix an edge of $S$, so that $\tilde{A}$ would be abelian, contradicting that $A \neq 1$ and $Z(A) = 1$. 
    \end{proof}
    
    By Lemma \ref{lem: tilde A fixes a unique white vertex}, we have that for each white vertex $u \in V(T_C)$, we have that $\Gamma_u$ fixes a unique white vertex of $S$ (since $\Gamma_u = \tilde{A}^g$ for some $g \in \Gamma$, so that $\Gamma_u$ fixes the unique white vertex $u=gv$ where $v \in V(T_C)$ with Stab$_{\Gamma} (v) = \tilde{A}$). This gives us a $\Gamma$-equivariant map from the white vertices of $T_C$ to the white vertices of $S$. Similarly, each white vertex stabilizer $\Gamma_v$ of $S$ is elliptic on $T_C$, and it must fix a unique white vertex since otherwise we would obtain $H \oplus \Z < B \rtimes \langle t \rangle$ for some $H < G$ with $Z(H) = 1$, which would imply by taking projections to $G$ that $H < B$, contradicting that $H \neq 1$ and $Z(H) = 1$ (note that $Z(A) = 1$ since $A$ is non-abelian and finitely generated CSA groups are commutative transitive (see \cite{CT_CSA}), meaning that commutation is a transitive relation among non-trivial elements). Thus, we obtain a $\Gamma$-equivariant map from the white vertices of $S$ to the white vertices of $T_C$, and this map is inverse to the map from the white vertices of $T_C$ to the white vertices of $S$, since $\tilde{A} \leq \tilde{A}^g$ implies $\tilde{A} = \tilde{A}^g$ for any $g \in \Gamma$. 

    We now extend this map to a tree isomorphism $T_C \to S$.
    
    \begin{lemma}
        \label{lem: Tree of cylinders map is a tree iso}
        For each pair of white vertices $u,v$ of $T_C$ separated by a single black vertex $b$ we have that $u', v'$ are separated by a single black vertex $b'$ of $S$.
    \end{lemma}

    \begin{proof}[Proof of Lemma \ref{lem: Tree of cylinders map is a tree iso}]
        If this were not the case, then the segment $[u',v']$ of length 2 in $T_C$ would be sent to a segment of length at least 4 in $S$, which has pointwise stabilizer trivial by 2-acylindricity of $S$. However, the segment $[u,v]$ of $T_C$ always has non-trivial pointwise stabilizer. Indeed, since $T_C$ has at most two edge orbits $\Gamma [e^{-1} \tilde{A}, \tilde{B}]$, $\Gamma [\tilde{B}, \tilde{A}]$, up to $\Gamma$-translation, we have $[u,v] = [e^{-1} \tilde{A}, \tilde{B}] \cup g[\tilde{B}, \tilde{A}]$ or $g[e^{-1} \tilde{A}, \tilde{B}] \cup [\tilde{B}, \tilde{A}]$ for some $g \in \tilde{B}$. Each of theses segments has pointwise stabilizer $\tilde{D} \cap \tilde{D}^g \geq D$ since $g \in \tilde{B}$. 
    \end{proof}

    Therefore, by Lemma \ref{lem: Tree of cylinders map is a tree iso}$, [u',v']$ has a single black vertex. Thus, the $\Gamma$-equivariant bijection from white vertices of $T_C$ to white vertices of $S$ extends to a unique $\Gamma$-equivariant bijective tree map $T_C \to S$. 

    We conclude that $T_C$ and $S$ are isomorphic as $\Gamma$-trees.
    
\end{proof}

\begin{lemma}
    \label{lem: normalizer of A tilde for non-separating}
    Let $G = A*_{D^e = C}$ be a maximal splitting of a finitely generated torsion free CSA group $G$ with respect to a Dehn twist $\delta$ as above. Let $B$ be the maximal abelian subgroup in $G$ containing $D$. Let $\Gamma = G \rtimes_{\delta} \langle t \rangle$. If $\gamma \in \Gamma$ is such that $\tilde{A} \subseteq \tilde{A}^{\gamma}$, then $\gamma \in \tilde{A}$. Similarly, if $\gamma \in \Gamma$ is such that $\tilde{B} \subseteq \tilde{B}^{\gamma}$, then $\gamma \in \tilde{B}$. In particular, denoting $N_{\Gamma}(\tilde{A})$ the normalizer of $\tilde{A}$ in $\Gamma$, we have that $N_{\Gamma}(\tilde{A}) = \tilde{A}$ and $N_{\Gamma}(\tilde{B}) = \tilde{B}$.
\end{lemma}

\begin{proof}
    Since edge stabilizers of $T_C$ are abelian and since $\tilde{A}$ is non-abelian, the proof for $\tilde{A}$ follows exactly as in the proof of Lemma \ref{lem:normalizer of A tilde}. For $\tilde{B}$, if $\tilde{B}$ is non-abelian (i.e.\ $\delta$ acts non-trivially on $B$), then again the result follows exactly as in the proof of Lemma \ref{lem:normalizer of A tilde}. If $\tilde{B}$ is abelian (i.e.\ $\delta$ acts trivially on $B$), then $B \subseteq A$ and since $D$ is maximal abelian in $A$, we obtain that $D = B$ and so $\tilde{D} = \tilde{B}$. If $\gamma \in \Gamma$ is such that $\tilde{B} \subseteq \tilde{B}^{\gamma}$, then $\langle \tilde{D}, \tilde{D}^{\gamma}\rangle = \langle \tilde{B}, \tilde{B}^{\gamma}\rangle = \tilde{B}$ which is abelian. It was shown in the proof of Theorem \ref{thm: Tree of cylinders of mapping torus} that this implies $\gamma \in \tilde{D} = \tilde{B}$.
\end{proof}

\begin{lemma}
    \label{lem: fibre and orientation preserving iso}
    Let $G$ be a finitely generated torsion-free CSA group. Let $\delta$ and $\delta'$ be a non-separating Dehn twists of $G$ with associated maximal splittings $G = A*_{D^e = C}$ for $\delta$ and $G = A'*_{D'^{e'} = C'}$ for $\delta'$, with $\delta(e) = ew_1$ and $\delta'(e') = e'w_1'$ for some $w_1 \in D \setminus 1$ and $w_1' \in D' \setminus 1$. Let $\sigma : D \to C$ and $\sigma' : D' \to C'$ be the associated isomorphisms from the above HNN extensions of $G$. Then $\delta$ and $\delta'$ are conjugate in $\Aut(G)$ if and only if there exists an isomorphism $\phi: A \to A'$ such that: 
    
    \begin{enumerate}
        \item $\phi(D) \subseteq D'$,
        \item $\phi(C) \subseteq C'$,
        \item $\phi(w_1) = w_1'$,
        \item $\sigma' \circ \phi \vert_D = \phi \circ \sigma$
    \end{enumerate}
    
\end{lemma}

\begin{proof}

Let $\Gamma$, $\Gamma'$ be the mapping torus corresponding to these two non-separating Dehn twists on $G$ and $G'$ respectively, i.e. $\Gamma = (A*_{D^e = C}) \rtimes_\delta \langle t \rangle$ and $\Gamma' = (A'*_{D'^{e'} = C'}) \rtimes_{\delta'} \langle t' \rangle$ and let $T_C$, $T_C'$ be the tree of cylinders of the associated splitting of the mapping torus $\Gamma = \tilde{A}*_{\tilde{D}^e = \tilde{C}}$ and $\Gamma' = \tilde{A'}*_{\tilde{D'}^{e'} = \tilde{C'}}$, respectively. Assume that $\delta$ and $\delta'$ are conjugate in $\Aut(G)$. Then by Lemma \ref{lem:mapping_tori_iso_implies_conjuguate}, there exists an isomorphism, $f: \Gamma \to \Gamma'$ such that $f|_G \in \Aut(G)$ and  $f(t)=t'$. The isomorphism $f$ induces a group action $._f$ of $\Gamma$ on the tree of cylinders $T_C'$: $g ._f x = f(g).x$, for all $g \in \Gamma$ and $x \in T_C'$. This makes $T_C'$ a $\Gamma$-tree that satisfies all the conditions of Lemma \ref{lem: canonicity lemma for tree of cylinders} as follows.

\begin{enumerate}
    \item The edge stabilizers, $\mathrm{Stab}_\Gamma(\tilde{z})$, are of the form $f^{-1}(\mathrm{Stab}_{\Gamma'}(\tilde{z}))$, for $\tilde{z} \in E(T_C')$, which is abelian because $\mathrm{Stab}_{\Gamma'}(\tilde{z})$ is abelian and $f$ is an isomorphism.

    \item White vertex stabilizers of $T_C'$ in $\Gamma'$ are of the form $(A' \oplus \langle t' \rangle)^g = A'^g \oplus \langle t'^g \rangle$, for some $g \in \Gamma'$. Note that $A'^g$ has trivial center in $G$. Therefore white vertex stabilizers of $T_C'$ in $\Gamma$ look like $f^{-1}(A'^g \oplus \langle t'^g \rangle) = f^{-1}(A'^g) \oplus f^{-1}(\langle t'^g \rangle) \cong H \oplus \mathbb{Z}$ with $H < G$ having a trivial center because $f$ is an isomorphism mapping $G$ to $G$.

    \item Black vertices in $T_C'$ have stabilizers in $\Gamma'$ of the form $(B' \rtimes \langle t' \rangle)^g = B'^g \rtimes \langle t'^g \rangle$, for some $g \in \Gamma'$, where $B'$ is the maximal abelian subgroup of $G$ containing $D'$. Now black vertex stabilizers of $T_C'$ in $\Gamma$ are therefore of the form $f^{-1}(B'^g \rtimes \langle t'^g \rangle) = f^{-1}(B'^g) \rtimes f^{-1}(\langle t'^g \rangle) \cong K \rtimes \mathbb{Z}$ with $K < G$ being abelian.

    \item Because $._f$ is an induced action from an isomorphism $f: \Gamma \to \Gamma'$ with $f|_G \in \Aut(G)$ and  $f(t)=t'$, and the action of $\Gamma'$ on $T_C'$ is $2$-acylindrical (due to Lemma \ref{thm:tree of cylinders max splitting is acylindrical}), the action of $\Gamma$ on $T_C'$ is $2$-acylindrical as well.
\end{enumerate}

From Lemma \ref{lem: canonicity lemma for tree of cylinders} it follows that there is a $\Gamma$-equivariant isomorphism between the two tree of cylinders $T_C$ and $T_C'$. 

By Lemma \ref{lem: tilde A fixes a unique white vertex}, we have that $\tilde{A}$ fixes a unique white vertex $\gamma \tilde{A}'$ of $T_C'$ (where we can assume $\gamma \in G$), and $\tilde{A}'$ also fixes a unique white vertex of $T_C$. Using the same argument as in the proof of Lemma \ref{lem:upgrade_to_fiber_preserving}, we obtain $f(\tilde{A}) = (\tilde{A}')^{\gamma}$, and similarly $f(\tilde{B}) = (\tilde{B}')^{\eta_1}$, $f(\tilde{B}^e) = (\tilde{B}')^{\eta_2}$ for some $\eta_1,\eta_2 \in G$.

Since $[\tilde{A}, \tilde{B}]$ and $[\tilde{A}, e\tilde{B}]$ are edges in $T_C$ that are, in general, in different $\Gamma$ orbits, by Lemma \ref{lem: Tree of cylinders map is a tree iso}, we also have that $[\gamma \tilde{A}', \eta_1 \tilde{B}']$ and $[\gamma \tilde{A}', \eta_2 \tilde{B}']$ are edges in $T_C'$ that are in different $\Gamma$ orbits. Thus, we have $\eta_1 \tilde{B}' = \alpha_1 \gamma \tilde{B}'$ and $\eta_2 \tilde{B}' = \alpha_2 \gamma e' \tilde{B}'$, or $\eta_1 \tilde{B}' = \alpha_1 \gamma e' \tilde{B}'$ and $\eta_2 \tilde{B}' = \alpha_2 \gamma \tilde{B}'$ for some $\alpha_1, \alpha_2 \in (\tilde{A}')^{\gamma}$. Up to replacing $e$ by $e^{-1}$ and $D$ with $C$, we can assume that the former case occurs. We can assume $\eta_1 = \gamma$ and $\eta_2 = \gamma e'$ as in Lemma \ref{lem:upgrade_to_fiber_preserving} and so $f(\tilde{A}) = (\tilde{A}')^{\gamma}$, $f(\tilde{B}) = (\tilde{B}')^{\gamma}$ and $f(\tilde{B}^e) = (\tilde{B}')^{\gamma e'}$. Define $\phi = \mathrm{Ad}(\gamma^{-1}) f$. Then $\phi(\tilde{A}) = \tilde{A}', \phi(\tilde{B}) = \tilde{B}'$ and $\phi(\tilde{B}^e) = (\tilde{B}')^{e'}$. These imply that $\phi(\tilde{D}) = \tilde{D}'$ and $\phi(\tilde{C}) = \tilde{C}'$ (since $\tilde{D} = \tilde{A} \cap \tilde{B}$ and $\tilde{C} = \tilde{A} \cap \tilde{B}^e$, and similarly for $\tilde{D}'$ and $\tilde{C}'$). Since $\phi (G) = G$, we obtain that $\phi(A) = A'$, $\phi(B) = B'$, $\phi(C) = C'$ and $\phi(D) = D'$.

Since $\phi(\tilde{B)} = \tilde{B}'$ and $\phi(\tilde{B}^e) = (\tilde{B}')^{e'}$, we obtain that $(e')^{-1}\phi(e) \in N_{\Gamma'}(\tilde{B}') \cap G = \tilde{B}' \cap G = B'$ (where $N_{\Gamma'}(\tilde{B}') = \tilde{B}'$ by Lemma \ref{lem: normalizer of A tilde for non-separating}), so that $\phi(e) = e' b'$ for some $b' \in B'$. Up to re-defining $e'$ to be $e'b'$ and re-defining $w_1'$ to be $(w_1')^{n}$ for some $n \in \Z$ (which does not change the group $G$ nor the way $\delta'$ acts on $G$, hence does not change $\Gamma'$), we can thus assume that $\phi(e) = e'$. Note also that we can arrange for $\phi(t) = t'$ since $\langle t \rangle = Z(\tilde{A})$ (respectively, $\langle t' \rangle = Z(\tilde{A}'$)) and $\phi(\tilde{A}) = \tilde{A}'$. Now we have $\phi(tet^{-1}) = \phi(ew_1)$, and since $\phi(t) = t'$ and $\phi(e) = e'$, we obtain $\phi(w_1) = w_1'$.

Lastly, we check that $\sigma' \circ \phi \vert_D = \phi \circ \sigma$. For each $d \in D$, we have $\sigma(d) = ede^{-1}$, so $\phi(\sigma(d)) = \phi(ede^{-1}) = e' \phi(d) (e')^{-1}$. On the other hand, we have $\sigma'(\phi(d)) = e' \phi(d) (e')^{-1}$. Thus, we see that $\sigma' \circ \phi \vert_D = \phi \circ \sigma$. The map $\phi \vert_A$ is thus the required isomorphism.

\vspace{0.2in}

Conversely, let us assume that there is an isomorphism $\phi: A \to A'$ such that $\phi(D) \subseteq D'$, $\phi(C) \subseteq C'$, $\phi(w_1) = w_1'$ and  $\sigma' \circ \phi|_D = \phi \circ \sigma$. To show that $\delta$ is conjugate to $\delta'$ in $\Aut(G)$, by Lemma \ref{lem:mapping_tori_iso_implies_conjuguate} it suffices to prove that there exists an isomorphism $F: \Gamma \to \Gamma'$ such that $F|_G \in \Aut(G)$ and $F(t) = t'$. Define $F: \Gamma \to \Gamma'$ by $F(t) = t'$, $F(e) = e'$ and $F|_A = \phi$. To show that $F$ is an isomorphism, the only thing that we need to check is if the this map is compatible with the edge maps, i.e. $\tilde{\sigma} \circ F|_{\tilde{D}} = F \circ \tilde{\sigma}$, where $\tilde{\sigma}: \tilde{D} \to \tilde{C}$ and $\tilde{\sigma'}: \tilde{D'} \to \tilde{C'}$ are the edge group isomorphisms associated to the HNN splittings of $\Gamma$ and $\Gamma'$ respectively with the following properties: $\tilde{\sigma}|_{\tilde{D}} = \sigma: D \to C$ and $\tilde{\sigma'}|_{\tilde{D'}} = \sigma': D' \to C'$ are the edge maps corresponding to the splittings $G = A*_{D^e = c}$ and $G = A'*_{D'^{e'} = C'}$ respectively, $\tilde{\sigma}(t) = tw_2^{-1}$ and $\tilde{\sigma'}(t') = t'w_2'^{-1}$. Now $\tilde{\sigma} \circ F|_{\tilde{D}}(t) = \tilde{\sigma'}(t') = t'w_2'^{-1}$ and $F \circ \tilde{\sigma} (t) = F(tw_2^{-1})=t'w_2'^{-1}$ (as $F(t) = t'$ and $F|_A = \phi$ by definition), and note that $\phi(w_2) = w_2'$ by conditions (3) and (4). This implies that $\tilde{\sigma} \circ F|_{\tilde{D}}(t) = F \circ \tilde{\sigma}(t)$. Next, let $d \in D$. Then

\begin{align*}
(\tilde{\sigma'} \circ F|_{\tilde{D}})(d) &= (\tilde{\sigma'} \circ \phi)(d) \\
&= (\sigma' \circ \phi)(d) \hspace {0.2in} [\phi(D) \subseteq D']\\
&= (\phi \circ \sigma)(d) \\
&= (F \circ \tilde{\sigma})(d) \hspace {0.2in} [\tilde{\sigma} \vert_D = \sigma \text{ and } \tilde{F} \vert_C = \phi_C]\\
\end{align*}

\noindent This implies that $(\tilde{\sigma} \circ F|_{\tilde{D}})|_D = (F \circ \tilde{\sigma})|_D$.

\begin{rmk}
    \label{rmk: Simplified conditions for non-sep suff conditions lemma}
    Note that conditions (1) and (2) in Lemma \ref{lem: fibre and orientation preserving iso} follow from conditions (3) and (4) and Lemma \ref{thm: Enlarging the splitting}, since $C$ and $D$ are maximal abelian in $A$, and hence are equal to $C_A(w_2)$ and $C_A(w_1)$, respectively. 
\end{rmk}
    
\end{proof}

\subsection{Algorithm to solve conjugacy of non-separating Dehn twists}
\label{section: algorithm for non-separating Dehn twists}

Let $G$ be a free product of finitely many finitely generated free abelian groups. Let $\delta$ and $\delta'$ be non-separating Dehn twists of $G$. Denote $\Gamma = G \rtimes_{\delta} \langle t \rangle$ and $\Gamma' = G' \rtimes_{\delta'} \langle t' \rangle$ the associated mapping tori. By Lemma \ref{thm: Enlarging the splitting}, there exist maximal splittings 

$$G = A*_{D^e = C}$$

and 

$$G = A'*_{(D')^{e'} = C'}$$

with respect to $\delta$ and $\delta'$. Let $w_1 \in D$ and $w_1' \in D'$ be such that $\delta(e) = ew_1$ and $\delta'(e') = e'w_1'$ and denote $w_2 = w_1^e$ and $w_2' = w_2^e$. From the discussion following the proof of Lemma \ref{thm: Enlarging the splitting}, we can assume that either $D$ and $C$ are cyclic, or $D=C$ or that $D*C$ is a free factor of $A$, and similarly with $D',C' < A'$. 

These splittings yield corresponding splittings of $\Gamma$ and $\Gamma'$: 

$$\Gamma = \tilde{A}*_{\tilde{D}^e = \tilde{C}}$$

and 

$$\Gamma' = \tilde{A'}*_{\tilde{(D')}^{e'} = \tilde{C'}}$$

where $\tilde{A} = A \oplus \langle t \rangle, \tilde{A'} = A' \oplus \langle t' \rangle$, etc. Let $T_C$ and $T_C'$ denote the tree of cylinders of the Bass--Serre trees $T$ and $T'$ corresponding to the above splittings of $\Gamma$ and $\Gamma'$. Denote also $B$ (respectively, $B'$) the maximal abelian subgroup of $G$ containing $D$ (respectively, $D'$) and write $\tilde{B} = B \rtimes \langle t \rangle$ and $\tilde{B'} = B' \rtimes \langle t' \rangle$. The following lemma is well-known.

\begin{lemma}
    \label{lemma: Whitehead problem for fg free abelian groups}
    Let $v = (v_1,v_2, \ldots, v_n), v' = (v_1',v_2',\ldots,v_n') \in \Z^n$. Then there exists $\psi \in \Aut(\Z^n) = \mathrm{GL}_n \Z$ such that $\psi(v) = v'$ if and only if $\gcd(v_1,v_2,\ldots,v_n) = \gcd(v_1',v_2',\ldots,v_n')$. 
\end{lemma}
Let (*) denote the following algorithmic problem: 

Given $\sigma, \sigma' \in \mathrm{GL}_n \Z$, and $v,v' \in \Z^n$, decide if there exist $\psi_1, \psi_2 \in \mathrm{GL}_n \Z$ such that: 

\begin{align}
    &\psi_1(v) = v', \\
    &\psi_2 \sigma \psi_1^{-1} = \sigma'
\end{align}

By Lemma \ref{lemma: Whitehead problem for fg free abelian groups}, to solve the algorithmic problem (*), it suffices to run the Euclidean algorithm on the components of $v,v'$ to compute their gcds. If the gcds are not the same, output NO and stop. Otherwise, output YES and stop. The YES output gives the correct answer to the problem because if the gcds of the components of $v$ and $v'$ are the same, then by Lemma \ref{lemma: Whitehead problem for fg free abelian groups}, there exists $\psi_1 \in \mathrm{GL}_n \Z$ such that $\psi_1(v) = v'$, satisfying condition (1) of (*). We can then satisfy condition (2) by setting $\psi_2 = \sigma' \psi_1 \sigma^{-1} \in \mathrm{GL}_n \Z$.

\textbf{Step 1}: Starting from the semi-direct product presentations of $\Gamma$ and $\Gamma'$, apply Tietze transformations to each presentation to compute the presentations for $\Gamma$ and $\Gamma'$ corresponding to the Bass--Serre trees $T$ and $T'$ and the trees of cylinders $T_C$ and $T_C'$. Recall that the splittings corresponding to $T_C$ and $T_C'$ have vertex groups $\tilde{A}$ and $\tilde{B}$ with edge groups $\tilde{D}$ (respectively, $\tilde{A'}$ and $\tilde{B'}$ with edge groups $\tilde{D'}$), hence from these splittings we can compute $\tilde{A}, \tilde{D}$ (respectively, $\tilde{D'}$ and $\tilde{A'}$). We can also compute the words $w_1$ and $w_1'$ by computing $w_1 = e^{-1}tet^{-1}$ (respectively, $w_1' = e'^{-1}t'e't'^{-1}$). We can assume that these presentations of $\Gamma$ and $\Gamma'$ include presentations for $\tilde{A} = A \oplus \langle t \rangle$ and $\tilde{A'} = A' \oplus \langle t' \rangle$ as direct products, from which we can extract presentations of $A$ and $A'$. We compute presentations for $D,C$ and $D',C'$ as well by extracting them from free product decompositions of $A$ and $A'$, or else computing the words $w_1$ and $w_2$ if $D,C$ are cyclic (similarly for $D',C'$). If the splittings for $T_C$ and $T_C'$ have different numbers of edge groups, then output NO and stop. Otherwise, proceed to Step 2.

\textbf{Step 2}: From the free product presentations for $A$ and $A'$, decide if $A \cong A'$ by comparing the list of ranks of the free factors of $A$ and $A'$ (recall that $A$ and $A'$ are finitely generated by Lemma \ref{lemma: A is finitely generated}). Also decide if $D \cong D'$ by comparing their ranks. If the list of ranks of free factors of $A$ and $A'$ are not the same or $D \ncong D'$, output NO and stop. Otherwise, extract free bases for $D$ and $D'$ from their presentations and proceed to Step 3.

\textbf{Step 3:} Identify $A$ with $A'$, $D$ with $D'$ and $C$ with $C'$. If $D$ and $D'$ are cyclic, use Lemma \ref{lemma: simultaneous whitehead algorithm} to decide if there exists an automorphism $\phi \in \Aut(A)$ such that $\phi(w_1) = w_1'$ and $\phi(w_2) = w_2'$. If such an automorphism exists, output YES and stop. Otherwise, output NO and stop. If $D$ and $D'$ are not cyclic, then fix isomorphisms $\alpha : D \to \Z^n$ and $\alpha' : D' \to \Z^n$ and isomorphisms $\beta : C \to \Z^n$, $\beta' : C' \to \Z^n$ coming from computing bases of $D$ and $D'$ and $C$ and $C'$, respectively. Denote $v = \alpha(w_1)$ and $v' = \alpha'(w_1')$. Using the Euclidean algorithm, compute the greatest common divisors $N:=\gcd(v_1,\ldots,v_n)$ and $N':= \gcd(v_1',\ldots,v_n')$. If $N=N'$, output YES and stop. Otherwise, output NO and stop.

We now justify why the algorithm works, i.e.\ why it gives the correct output. First, when the algorithm outputs NO in Step 1, then $T_C / \Gamma$ and $T_C' / \Gamma'$ have different numbers of edge groups and hence by Lemma \ref{lem: canonicity lemma for tree of cylinders}, there cannot be a fibre and orientation preserving isomorphism $f: \Gamma \to \Gamma'$, since as observed in the proof of Lemma \ref{lem: fibre and orientation preserving iso}, such an isomorphism would induce a $\Gamma$-equivariant isomorphism between $T_C$ and $T_C'$. Thus, by Lemma \ref{lem:mapping_tori_iso_implies_conjuguate}, we have that $\delta$ and $\delta'$ are not conjugate. Next, when the algorithm outputs NO in Step 2, then $A$ and $A'$ are not isomorphic or $D$ and $D'$ are not isomorphic, and hence by Lemma \ref{lem: fibre and orientation preserving iso}, we have that $\delta$ and $\delta'$ are not conjugate, so the algorithm returns the correct output in this case. To justify that the algorithm works in Step 3, if $D$ and $D'$ are cyclic, then by Lemma \ref{lem: fibre and orientation preserving iso}, $\delta$ and $\delta'$ are conjugate if and only if such an automorphism $\phi$ exists. Otherwise, if $D$ and $D'$ are not cyclic, then by Lemma \ref{lemma: Whitehead problem for fg free abelian groups}, it suffices to show that the four conditions in Lemma \ref{lem: fibre and orientation preserving iso} are satisfied if and only if the algorithmic problem (*) gives positive output for suitable choices of $v,v' \in \Z^n$ and $\bar{\sigma}, \bar{\sigma'} \in \mathrm{GL}_n(\Z)$. Indeed let $\sigma : D \to C$ and $\sigma' : D' \to C'$ be the edge isomorphisms. Let $v = \alpha (w_1)$ and $v' = \alpha' (w_1')$ where $\alpha: D \to \mathbb{Z}^n$ and $\alpha': D' \to \mathbb{Z}^n$ are the chosen isomorphisms in the algorithm above and $w_1, w_1'$ are as mentioned in Lemma \ref{lem: fibre and orientation preserving iso}. Let $\bar{\sigma} = \beta \sigma \alpha^{-1}$ and $\bar{\sigma'} = \beta' \sigma' \alpha'^{-1}$. Now let us assume that (*) gives positive output for $v,v'$ and $\bar{\sigma}, \bar{\sigma'}$, i.e. there exist $\psi_1, \psi_2 \in \mathrm{GL}_n\mathbb{Z}$ such that $\psi_1(v) = v'$ and $\psi_2 \bar{\sigma} \psi_1^{-1} = \bar{\sigma'}$. We define $\psi_D : D \to D'$ by $\psi_D = \alpha'^{-1} \psi_1 \alpha$ and $\psi_C : C \to C'$ by $\psi_C = \sigma' \psi_D \sigma^{-1}$. In the case that $D \neq C$ (equivalently, $\sigma \neq id_D$), then also $D' \neq C'$ (equivalently, $\sigma' \neq id_{D'}$) by Step 1 of the algorithm. We then define $\phi : A \to A'$ by using the isomorphisms $\psi_D : D \to D'$, $\psi_C : C \to C'$ and $\psi_R : R \to R'$, where $A \cong D*C*R \cong D'*C'*R'$ (note that an isomorphism $\psi_R$ exists because the two free factor presentations of $A$ and $A'$ are isomorphic). To be more precise, we can define $\phi = \psi_D * \psi_C * \psi_R$. We claim that $\phi$ satisfies the four conditions in Lemma \ref{lem: fibre and orientation preserving iso}:

\begin{itemize}
    \item $\phi|_D = \psi_D$ and therefore $\phi(D) = D'$ by definition.

    \item Similarly, $\phi(C) = C'$ by definition.

    \item
    \begin{align*}
        \phi(w_1) &= \psi_D (w_1)\\
        &= (\alpha'^{-1} \psi_1 \alpha)(w_1)\\
        &= (\alpha'^{-1} \psi_1)(v)\\
        &= \alpha'^{-1} (v')\\
        &= \alpha'^{-1} (\alpha'(w_1'))\\
        &= w_1'
    \end{align*}
    
    \item $\phi \sigma = \psi_C \sigma = \sigma' \psi_D \sigma^{-1} \sigma = \sigma' \psi_D = \sigma' \phi|_D$.
    
\end{itemize}

Conversely, let $\phi : A \to A'$ be the isomorphism as in Lemma \ref{lem: fibre and orientation preserving iso} satisfying the four conditions. Then we can take $\psi_D = \phi|_D : D \to D'$, $\psi_C = \phi|_C : C \to C'$. Define $\psi_1 = \alpha' \psi_D \alpha^{-1}$ and $\psi_2 = \bar{\sigma'} \psi_1^{-1} \bar{\sigma}^{-1}$. Then

\begin{align*}
    \psi_1(v) &= \psi_1(\alpha(w_1))\\
    &= (\alpha' \psi_D \alpha^{-1})(\alpha(w_1))\\
    &= (\alpha' \psi_D)(w_1)\\
    &= (\alpha' \phi)(w_1)\\
    &= \alpha'(w_1')\\
    &= v'
\end{align*}

\noindent and $\bar{\sigma'} = \psi_2 \bar{\sigma} \psi_1$ by definition and therefore (*) gives positive output for $v,v'$ and $\bar{\sigma}$ and $\bar{\sigma'}$. 

In the case that $D = C$, then $D' = C'$ by Step 1 and $\sigma = id_D, \sigma' = id_{D'}$, and $A = D * R \cong D' * R'$ where $R \cong R'$. Put $\phi = \psi_D * \psi_R$ for a fixed isomorphism $\psi_R : R \to R'$. Then similarly to above, we have that the four conditions of Lemma \ref{lem: fibre and orientation preserving iso} are satisfied (note that the condition $\phi \sigma = \sigma' \phi \vert_D$ is trivial since $\sigma = id_D$ and $\sigma' = id_{D'}$).

Conversely, suppose that $\phi : A \to A'$ is an isomorphism as in Lemma \ref{lem: fibre and orientation preserving iso} satisfying the four conditions. Then we take as above $\psi_D = \phi|_D : D \to D'$ and define $\psi_1 = \alpha' \psi_D \alpha^{-1}$ (note that $\bar{\sigma} = \bar{\sigma'} = id$). As above, we obtain $\psi_1(v) = v'$, hence (*) gives positive output for $v,v'$ and $\bar{\sigma} = \bar{\sigma'} = id$.

Therefore, the four conditions of Lemma \ref{lem: fibre and orientation preserving iso} hold if and only if (*) gives positive output for $v,v'$ and $\bar{\sigma}, \bar{\sigma'}$. 

Lastly, by Lemma \ref{lemma: Whitehead problem for fg free abelian groups}, the problem (*) gives positive output for $v,v' \in \Z^n$ and $\bar{\sigma}, \bar{\sigma'} \in \mathrm{GL}_n(\Z)$ if and only if $N = N'$, so the output at the end of Step 3 is correct. 

\begin{rmk}
    \label{rmk: Using generalized whitehead also works}
    In the case of when $D$ and $D'$ are free factors of $A$ and $A'$, respectively, to decide if there exists an isomorphism $\phi_D : D \to D'$ taking $w_1$ to $w_1'$, we could have also used Lemma \ref{lemma: simultaneous whitehead algorithm} applied to $D \cong D'$ and $w_1, w_1'$. However, our solution of using (*) is far more elementary, hence why we choose to use (*) instead of Lemma \ref{lemma: simultaneous whitehead algorithm}.
\end{rmk}

\section{The algorithm to solve the conjugacy problem for general separating or non-separating Dehn twists}

In this section, we develop an algorithm to solve the conjugacy problem for two arbitrary Dehn twists $\delta$ and $\delta'$ of a finitely generated group $G$ which splits as a free product of free abelian groups.

\begin{lemma}
    \label{lemma: Dehn twists cannot be both separating and non-separating}
    Let $G$ be a finitely generated group which splits as a free product of free abelian groups. Let $\delta$ be a non-trivial Dehn twist of $G$. Then $\delta$ cannot be both separating and non-separating. 
\end{lemma}

\begin{proof}
    Suppose that $\delta$ was both a separating and non-separating Dehn twist of $G$. Then by Theorem \ref{lemma: enlarging the splitting in separating case} we can split $G$ on the one hand as: 

    $$G = A_1 *_{C_1} B_1$$ 

    with $C_1$ maximal abelian in $A_1, B_1$, with $\delta$ acting as the identity on $A_1$ and conjugation by some $c \in C_1$ on $B_1$. By Theorem \ref{lemma: enlarging the splitting in separating case}, we have that $A_1 = \mathrm{Fix}(\delta)$. 

    On the other hand, $G$ also splits as:

    $$G = {A_2}*_{D_2^e = C_2}$$ 
    
     By Lemma \ref{thm: Enlarging the splitting}, we can assume that $A_2 = \mathrm{Fix}(\delta)$ and that $D_2, C_2 < A_2$ are maximal abelian. Write $\delta(e) = ed$ for some $d \in D_2$. We denote $A = A_1 =A_2$.

    Let $\Gamma = G \rtimes_{\delta} \langle t \rangle$ denote the mapping torus.

    We must have that $A$ and $B_1$ are non-abelian. Indeed, suppose that $A$ is abelian. Then $C_1 = A$ and so $G = B_1$ and $\delta$ is an inner automorphism of $G$, of the form $\delta = \Ad(c)$ for some $c \in A$. Also, from the HNN splitting of $G$ we obtain $A = D_2 = C_2$ and $e$ commutes with $A$, so $G = A \oplus \langle e \rangle$. But then $G$ is abelian and so $\delta$ is trivial. However, this contradicts that $\delta$ is non-trivial. Therefore, $A_1$ cannot be abelian if $\delta$ is non-trivial.

    Now suppose that $B_1$ is abelian. Then $C_1 = B_1$ and $G = A_1$, and so $\delta$ is trivial, a contradiction. Therefore, we must have that $A_1$ and $B_1$ are non-abelian if $\delta$ is non-trivial and is both separating and non-separating. 

    Then by Lemma \ref{lemma: acylindricity of separating mapping torus}, we have a canonical 2-acylindrical splitting $T$ of $\Gamma$ as an amalgam. 

    $$\Gamma = \tilde{A_1} *_{\tilde{C_1}}  \tilde{B_1}$$

    and by Theorem \ref{thm:tree of cylinders max splitting is acylindrical}, we have another canonical 2-acylindrical splitting $T_C$ of of $\Gamma$ coming from the tree of cylinders construction. 

    We have that $T_C$ satisfies the conditions of Lemma \ref{lem: Canonicity lemma for separating Dehn twists}, and hence $T$ and $T_C$ are isomorphic $\Gamma$-trees. However, this is a contradiction, as $T_C/\Gamma$ has solvable vertex group $\tilde{B_2} = B_2 \rtimes \langle t \rangle$ while the vertex groups of $T/\Gamma$ contain $F_2$ (since $A_1, B_1$ contain $F_2$, being free products of at least 2 free abelian groups). 

\end{proof}

\begin{lemma}
    \label{lem: conjugate Dehn twists either both separating or non-separating}
    Let $G$ be a finitely generated group which splits as a free product of free abelian groups. Let $\delta, \delta'$ be non-trivial conjugate Dehn twists of $G$. Then either $\delta$ and $\delta'$ are both separating or both non-separating. 
\end{lemma}

\begin{proof}
    Suppose that $\delta$ is separating and $\delta'$ is non-separating. Write associated splittings $G = A *_C B$ for $\delta$ and $G = J*_{H^e=K}$ for $\delta'$. Let $f \in \Aut(G)$ be such that $\delta' = \delta^f$ in $\Aut(G)$. Then $\delta$ is also non-separating with respect to the splitting $G = f^{-1}(J)_{*f^{-1}(H)^{f^{-1}(e)} = f^{-1}(K)}$. Since $\delta$ is non-trivial, this contradicts Lemma \ref{lemma: Dehn twists cannot be both separating and non-separating}. 
\end{proof}

\begin{lemma}
    \label{lemma: conjugacy problem for unimodular matrices}
    Let $A$ be a finitely generated free abelian group with basis $(e_1,\ldots,e_{n+1})$. Let $\delta, \delta' \in \Aut(A)$ be of the form $\delta(e_i) = \delta'(e_i) = e_i$ for each $i = 1,\ldots,n$ and $\delta(e_{n+1}) = v + e_{n+1}, \delta'(e_{n+1}) = v' + e_{n+1}$ for some $v = (v_1,\ldots,v_n),v'=(v_1',\ldots,v_n') \in \langle e_1,\ldots, e_{n} \rangle$. Then $\delta$ and $\delta'$ are conjugate in $\Aut(A)$ if and only if $\gcd(v_1,\ldots,v_n) = \gcd(v_1',\ldots,v_n')$.
\end{lemma}

\begin{proof}

    We identify $A \cong \Z^{n+1}$ and $\Aut(A) = \mathrm{GL}_{n+1}(\Z)$ via the choice of basis $(e_1,\ldots,e_{n+1})$. 

    If $\gcd(v_1,\ldots,v_n) = \gcd(v_1',\ldots,v_n')$, then by Lemma \ref{lemma: Whitehead problem for fg free abelian groups}, we have that there exists $\psi \in \mathrm{GL}_n(\Z)$ such that $\psi(v) = v'$. Consider the block matrix $M \in \mathrm{GL}_{n+1}(\Z)$ with $\psi$ in the upper left $n \times n$ block and last row and column $(0,0,\ldots,0,1)$. We then see that: 

    \begin{align*}
        M \delta &= \begin{bmatrix}
        \psi & 0 \\
        0 & 1
    \end{bmatrix} \begin{bmatrix}
        Id & v \\
        0 & 1
    \end{bmatrix} \\
    &= \begin{bmatrix}
        \psi & \psi(v) \\
        0 & 1
    \end{bmatrix} \\
    &= \begin{bmatrix}
        \psi & v' \\
        0 & 1
    \end{bmatrix} \\
    &= \begin{bmatrix}
        Id & v' \\
        0 & 1
    \end{bmatrix} \begin{bmatrix}
        \psi & 0 \\
        0 & 1
    \end{bmatrix} \\
    &= \delta' M
    \end{align*}

    So that $\delta$ and $\delta'$ are conjugate by $M$ in $\mathrm{GL}_{n+1}(\Z)$.

    Conversely, suppose that $\delta' = \delta^M$ for some $M \in \mathrm{GL}_{n+1}(\Z)$. Then for each $i=1,\ldots,n$ we have: 

    $$M \delta M^{-1}(e_i) = \delta'(e_i) = e_i$$

    so that $\delta M^{-1}(e_i) = M^{-1}(e_i)$ for all $i=1,\ldots,n$. We may assume that either both $\gcd(v_1,\ldots,v_n)$ and $ \gcd(v_1',\ldots,v_n')$ are zero or non-zero, since each of these gcd's are zero if and only if the corresponding automorphism $\delta, \delta'$ is trivial, in which case they are conjugate if and only if they are both trivial. Hence, the fixed point set of $\delta$ and $\delta'$ is $\langle e_1, \ldots, e_n \rangle$, and $M$ preserves this fixed point set. Thus, $M$ is block diagonal of the form: 

    $$M = \begin{bmatrix}
        \psi & x \\
        0 & x_{n+1}
    \end{bmatrix}$$

    with $x \in \Z^n$ and $x_{n+1} \in \Z \setminus 0$.

    Then: 

    $$M \delta = \begin{bmatrix}
        \psi & x \\
        0 & x_{n+1}
    \end{bmatrix} \begin{bmatrix}
        Id & v \\
        0 & 1
    \end{bmatrix} = \begin{bmatrix}
        \psi & \psi(v) + x \\
        0 & x_{n+1}
    \end{bmatrix}$$

    and

    $$\delta' M = \begin{bmatrix}
        Id & v' \\
        0 & 1
    \end{bmatrix} \begin{bmatrix}
        \psi & x \\
        0 & x_{n+1}
    \end{bmatrix} = \begin{bmatrix}
        \psi & x + v'x_{n+1} \\
        0 & x_{n+1}
    \end{bmatrix}$$

    Since $M \delta = \delta' M$, we obtain $\psi(v) + x = x + v'x_{n+1}$, and so $\psi(v) = v' x_{n+1}$. Since $\psi$ is an automorphism, this yields $\gcd(v_1,\ldots,v_n) = \gcd(v_1',\ldots,v_n') x_{n+1}$, so that $\gcd(v_1',\ldots,v_n')$ divides $\gcd(v_1,\ldots,v_n)$. Repeating the above argument with the equality $M^{-1} \delta' = \delta M^{-1}$, we obtain that $\gcd(v_1,\ldots,v_n)$ divides $\gcd(v_1',\ldots,v_n')$, so that $\gcd(v_1,\ldots,v_n)=\gcd(v_1',\ldots,v_n')$, since gcd is non-negative by definition. 
\end{proof}

\begin{lemma}
    \label{lem: Alternatives on canonical splittings of mapping torus}
    Let $G$ be a finitely generated group that is a free product of free abelian groups. Let $\delta$ be a Dehn twist of $G$. Denote $\Gamma = G \rtimes_{\delta} \langle t \rangle$ the associated mapping torus of $\delta$. Then exactly one of the following is true: 

    \begin{enumerate}
        \item $\Gamma$ decomposes as a direct product $\Gamma = G \oplus \langle gt \rangle$ for some $g \in G$. This case corresponds to when $\delta = \Ad(g^{-1})$ is an inner automorphism of $G$.
        \item $\Gamma$ decomposes as $\Gamma = (A \oplus \langle e \rangle) \rtimes \langle t \rangle$ where $A$ is abelian, $t$ centralizes $A$ and $tet^{-1} = ea_0$ for some $a_0 \in A \setminus \{1\}$. This corresponds to the case when $G$ is abelian and $\delta$ is a non-separating Dehn twist of $G$.
        \item $\Gamma$ splits as a graph of groups with 2-acylindrical Bass--Serre tree and with 2 non-abelian vertices and a single abelian edge group, as in Section \ref{section: separating Dehn twists}. This case corresponds to when $\delta$ is a non-trivial separating Dehn twist of a non-trivial amalgamated free product splitting of $G$. 
        \item $\Gamma$ splits as a graph of groups with 2-acylindrical Bass--Serre tree and with 2 vertex groups, one of which is non-abelian and the other is abelian-by-cyclic, with at most 2 edge groups which are both abelian, as in Section \ref{section: tree of cylinders for non-separating Dehn twist}. This corresponds to the case when $\delta$ is a non-trivial non-separating Dehn twist of $G$. 
    \end{enumerate}
\end{lemma}

\begin{proof}
    If $\delta$ is a Dehn twist of a non-trivial splitting of $G$, then by Sections \ref{section: separating Dehn twists} and Section \ref{section: tree of cylinders for non-separating Dehn twist}, $\Gamma$ splits as items (3) or (4). By Lemma \ref{lemma: Dehn twists cannot be both separating and non-separating}, items (3) and (4) cannot both hold. Also, (1) and (2) cannot hold since Proposition \ref{prop: properties of groups acting acylindrically on trees} and items (3) and (4) each imply that $\Gamma$ has finite centre, while (1) and (2) each imply that $\Gamma$ has infinite centre. 

    If $\delta$ is a Dehn twist of a trivial splitting of $G$, then if $G$ is non-abelian, this spliting must be a trivial amalgam splitting and $\delta$ is then an inner automorphism of $G$. Writing $\delta = \Ad(g)$ for some $g \in G$, we then have that $\Gamma = G \oplus \langle g^{-1}t \rangle$, so item (1) holds. Items (3) and (4) cannot hold since then by Proposition \ref{prop: properties of groups acting acylindrically on trees}, $\Gamma$ has finite centre $Z(\Gamma)$, which contradicts that $gt \in Z(\Gamma)$ and has infinite order. Item (2) cannot hold by Lemma \ref{lemma: Dehn twists cannot be both separating and non-separating}. If $G$ is abelian and $\delta$ is a non-trivial Dehn twist of a trivial splitting of $G$ (if $\delta$ is trivial then it is an inner automorphism and (1) holds), then $G$ splits as a direct product HNN extension $G=A \oplus \langle e \rangle$  for some abelian subgroup $A$ of $G$ and $\delta$ is a unimodular automorphism of $G$, i.e.\ $\delta \vert_A = id_A$ and $\delta(e) = ea_0$ for some $a_0 \in A$. Therefore, item (2) holds. Item (1) does not hold because if $G$ is abelian and (1) holds then $\delta = id_G$ by Lemma \ref{lemma: Dehn twists cannot be both separating and non-separating}. Items (3) and (4) also do not hold since they imply that $\Gamma$ has finite centre by Proposition \ref{prop: properties of groups acting acylindrically on trees}, while (2) implies that $\Gamma$ has infinite centre.
\end{proof}

As a corollary of Lemma \ref{lemma: Dehn twists cannot be both separating and non-separating}, we obtain the following:

\begin{lemma}
    \label{lem: alternative is conjugacy invariant}
    Let $G$ be a finitely generated group that is a free product of free abelian groups. Assume that $G$ is non-abelian. Let $\delta, \delta'$ be Dehn twists of $G$ that are conjugate in $\Aut(G)$. Then $\delta$ and $\delta'$ satisfy the same item of Lemma \ref{lem: Alternatives on canonical splittings of mapping torus}.  
\end{lemma}

\begin{proof}
    By Lemma \ref{lem:mapping_tori_iso_implies_conjuguate}, if $\delta$ and $\delta'$ are conjugate, then we have an isomorphism $f : \Gamma \to \Gamma'$ that maps $G$ to $G$ and $t$ to $t'$. This isomorphism $f$ then preserves the decomposition of $\Gamma$, i.e.\ $f$ induces a decomposition of $\Gamma'$ of the same type in Lemma \ref{lemma: Dehn twists cannot be both separating and non-separating} as $\Gamma$. 
    
\end{proof}

\subsection{The algorithm to solve conjugacy of arbitrary Dehn twists}: 

Let $G$ be a free product of finitely many finitely generated free abelian groups. Let $\delta$ and $\delta'$ be Dehn twists of $G$. Denote $\Gamma = G \rtimes_{\delta} \langle t \rangle$ and $\Gamma' = G' \rtimes_{\delta'} \langle t' \rangle$ the associated mapping tori.

\vs 

\textbf{Step 1}: Apply Tietze transformations to the semi-direct product presentations of the mapping tori $\Gamma$ and $\Gamma'$ to bring them each into one of the four forms of Lemma \ref{lem: Alternatives on canonical splittings of mapping torus}. 

\textbf{Step 2:} If $\Gamma$ and $\Gamma'$ fall into different categories in Lemma \ref{lem: alternative is conjugacy invariant}, then output NO and stop. Otherwise, proceed to Step 3.

\textbf{Step 3:} 

\begin{itemize}
    \item If $\Gamma$ and $\Gamma'$ both satisfy (1) of Lemma \ref{lem: Alternatives on canonical splittings of mapping torus}, then $\delta = \Ad(g)$ and $\delta' = \Ad(g')$ for some $g,g' \in G$. We take $g = 1$ (respectively, $g' = 1$) if $\delta$ (respectively, $\delta'$) is trivial. Apply the Whitehead algorithm in $G$ to decide if there exists $f \in \Aut(G)$ such that $f(g) = g'$ (c.f.\ \cite{Whitehead_for_toral_relatively_hyperbolic_groups} for the Whitehead algorithm for free products of free abelian groups). If there exists such an $f$, output YES and stop. Otherwise, output NO and stop. 
    \item If $\Gamma$ and $\Gamma'$ both satisfy (2) of Lemma \ref{lem: Alternatives on canonical splittings of mapping torus}, then compute a basis for $A$ and write $\delta$ and $\delta'$ in this basis as:

    $$\delta = \begin{bmatrix}
        Id & v \\
        0 & 1
    \end{bmatrix}, \delta' = \begin{bmatrix}
        Id & v' \\
        0 & 1
    \end{bmatrix}$$
    for $v = (v_1,\ldots,v_n)$ and $v'=(v_1',\ldots,v_n')$ in $\Z^n$. Compute $\gcd(v_1,\ldots,v_n)$ and $\gcd(v_1',\ldots,v_n')$ and check if these are the same. If not, output NO and stop. Otherwise, output YES and stop.
    \item If $\Gamma$ and $\Gamma'$ both satisfy (3) of Lemma \ref{lem: Alternatives on canonical splittings of mapping torus}, then apply the algorithm to decide conjugacy of $\delta$ and $\delta'$ from Section \ref{section: separating Dehn twists}.
    \item If $\Gamma$ and $\Gamma'$ both satisfy (4) of Lemma \ref{lem: Alternatives on canonical splittings of mapping torus}, then apply the algorithm to decide conjugacy of $\delta$ and $\delta'$ from Section \ref{section: algorithm for non-separating Dehn twists}.
\end{itemize}

We now justify why the above algorithm gives the correct output at each step. 

First, in the NO output in Step 2, we have that $\delta$ and $\delta'$ are not connjugate by Lemma \ref{lem: alternative is conjugacy invariant}, so the algorithm gives the correct output in this case. 

Next, in the first case in Step 3, the algorithm gives the correct output because two inner automorphisms $\delta = \Ad(g)$ and $\delta' = \Ad(g')$ are conjugate in $\Aut(G)$ if and only if there exists $f \in \Aut(G)$ such that $f(g) = g'$. Indeed, if such an $f$ exists, then for each $h \in G$, we have: 

\begin{align*}
    \Ad(g') (h) &= g'hg'^{-1} \\
    &= f(g) h f(g)^{-1} \\
    &= f(gf^{-1}(h) g^{-1}) \\
    &= f \Ad(g) f^{-1} (h)
\end{align*}

so that $\Ad(g') = f \Ad(g) f^{-1}$. 

Conversely, reversing the steps above shows that if $\Ad(g') = f \Ad(g) f^{-1}$ for some $f \in \Aut(G)$, then $f(g) = g'$. 

In the next case of Step 3, the algorithm gives the correct output by Lemma \ref{lemma: conjugacy problem for unimodular matrices}. 

In the third case and fourth cases of Step 3, the algorithm gives the correct output by the algorithms in Section \ref{section: algorithm for separating Dehn twists} and Section \ref{section: algorithm for non-separating Dehn twists}, respectively.

\bibliographystyle{plain}
\bibliography{biblio.bib} 
\end{document}